\documentclass[a4paper,11pt]{article}
\usepackage[latin1]{inputenc}
\usepackage[english]{babel}
\usepackage{amsmath}
\usepackage{amsfonts}
\usepackage{amssymb}
\usepackage{epsfig}
\usepackage{amsopn}
\usepackage{amsthm}
\usepackage{color}
\usepackage{graphicx}
\usepackage{subfigure}
\usepackage{enumerate}
\newtheorem{theorem}{Theorem}[section]

\newtheorem{lemma}[theorem]{Lemma}
\newtheorem{proposition}[theorem]{Proposition}
\newtheorem{definition}[theorem]{Definition}

\newtheorem*{theorem*}{Theorem}
\newtheorem*{lemma*}{Lemma}
\newtheorem*{remark*}{Remark}
\newtheorem*{definition*}{Definition}
\newtheorem*{proposition*}{Proposition}
\newtheorem*{corollary*}{Corollary}
\numberwithin{equation}{section}
\newcommand{\real}{\mathbb{R}}

\let\ced=\c         % cedilla
\def\a{\alpha}

\def\qed{\,\unskip\kern 6pt \penalty 500
\raise -2pt\hbox{\vrule \vbox to8pt{\hrule width 6pt
\vfill\hrule}\vrule}\par}
\definecolor{darkblue}{rgb}{0.05, .05, .65}
\definecolor{darkgreen}{rgb}{0.1, .65, .1}
\definecolor{darkred}{rgb}{0.8,0,0}
\newcommand{\beqn}{\begin{equation}}
\newcommand{\eeqn}{\end{equation}}
\newcommand{\bear}{\begin{eqnarray}}
\newcommand{\eear}{\end{eqnarray}}
\newcommand{\bean}{\begin{eqnarray*}}
\newcommand{\eean}{\end{eqnarray*}}
\begin{document}
%%%%%%%%%%%%%%%%%%%%%%%%%%%%%%%%%%%%%%%%%%%%%%%%%

%%%%%%%%%%%%%%%%%%%%%%%%%%%%%%%%%%%%%%%%%%%%%%%%%
\title{\huge \bf Blow up profiles for a quasilinear reaction-diffusion equation with weighted reaction}

\author{
\Large Razvan Gabriel Iagar\,\footnote{Instituto de Ciencias
Matem\'aticas (ICMAT), Nicol\'as Cabrera 13-15, Campus de
Cantoblanco, 28049, Madrid, Spain, \textit{e-mail:}
razvan.iagar@icmat.es}, \footnote{Institute of Mathematics of the
Romanian Academy, P.O. Box 1-764, RO-014700, Bucharest, Romania.}
\\[4pt] \Large Ariel S\'{a}nchez,\footnote{Departamento de Matem\'{a}tica
Aplicada, Ciencia e Ingenieria de los Materiales y Tecnologia
Electr\'onica, Universidad Rey Juan Carlos, M\'{o}stoles,
28933, Madrid, Spain, \textit{e-mail:} ariel.sanchez@urjc.es}\\
[4pt] }
\date{}
\maketitle

\begin{abstract}
We perform a thorough study of the blow up profiles associated to
the following second order reaction-diffusion equation with
non-homogeneous reaction:
$$
\partial_tu=\partial_{xx}(u^m) + |x|^{\sigma}u^p,
$$
in the range of exponents $1<p<m$ and $\sigma>0$. We classify blow up solutions in self-similar form, that are likely
to represent typical blow up patterns for general solutions. We thus show that the non-homogeneous coefficient $|x|^{\sigma}$ has a
strong influence on the qualitative aspects related to the finite time blow up. More precisely, for $\sigma\sim0$, blow
up profiles have similar behavior to the well-established profiles for the homogeneous case
$\sigma=0$, and typically \emph{global blow up} occurs, while for
$\sigma>0$ sufficiently large, there exist blow up profiles for which blow up \emph{occurs only at space infinity}, in strong contrast with the
homogeneous case. This work is a part of a larger program of
understanding the influence of unbounded weights on the blow up behavior for reaction-diffusion equations.
\end{abstract}

\


\smallskip


\section{Introduction}

In the present work, we deal with the phenomenon of blow up in
finite time for the following quasilinear reaction-diffusion
equation with a weighted reaction term:
\begin{equation}\label{eq1}
u_t=(u^m)_{xx}+|x|^{\sigma}u^p, \qquad u=u(x,t), \quad
(x,t)\in\real\times(0,T),
\end{equation}
in the range of exponents $1<p<m$ and $\sigma>0$, where, as usual,
the subscript notation in \eqref{eq1} indicates partial derivative
with respect to the time or space variable. We say that a solution
$u$ to \eqref{eq1} blows up in finite time if there exists
$T\in(0,\infty)$ such that $u(T)\not\in L^{\infty}(\real)$, but
$u(t)\in L^{\infty}(\real)$ for any $t\in(0,T)$. The time $T<\infty$
satisfying this property is known as the blow up time of $u$. Here
and in all the paper, we denote by $u(t)$ the map $x\mapsto u(x,t)$
for a fixed time $t\in[0,T]$.

The blow up phenomenon for the homogeneous reaction-diffusion
equation
\begin{equation}\label{eq2}
u_t=\Delta u^m+u^p,
\end{equation}
with either $m=1$ or $m>1$ is already well studied, cf. for example
the well-known books \cite{QS}, respectively \cite{S4}, the paper
\cite{GV}, and references therein. Meanwhile, due to its difficulty
introduced by the nonhomogeneous reaction term and the fact that
some important techniques such as translations or intersection
comparison do not work with unbounded weights, Eq. \eqref{eq1} is
much less studied. The main questions one addresses in the study of
the blow up phenomenon are:

\noindent $\bullet$ When does blow up occur (that is, for which
initial data)?

\noindent $\bullet$ In case blow up occurs at time $T\in(0,\infty)$,
what is the time scale (called rate) as $t\to T$?

\noindent $\bullet$ Where does blow up occurs? In which sets?

\noindent $\bullet$ How does blow up occurs? This raises the problem
of the "asymptotic" blow up behavior, that means, to which kind of
profile the solutions approach as $t\to T$.

Answers to most of these questions were given (at least partially)
for the homogeneous equation \eqref{eq2}. On the other hand, for Eq.
\eqref{eq1} and its $N$-dimensional form
\begin{equation}\label{eq3}
u_t=\Delta u^m + |x|^{\sigma}u^p, \qquad (x,t)\in\real^N\times(0,T),
\end{equation}
little is known. Some works were devoted to the semilinear case
$m=1$, establishing the critical (Fujita) exponent
$p_*=1+\frac{\sigma+2}{N}$ below which all the solutions with data
$u_0\in L^{\infty}(\real^N)$ blow up in finite time, and studying
the "life span" of solutions \cite{BK87, BL89, Pi97, Pi98}.
Afterwards, coupled systems of reaction-diffusion semilinear
equations with weighted reaction were also considered and conditions
for global existence or, on the contrary, finite time blow up were
established \cite{IU08}. More recently, some partial but quite
interesting results concerning blow up sets were established for the
semilinear case $m=1$, in particular concerning whether the origin
can or cannot be a blow up point (see for example the series of
papers \cite{GLS, GS11, GLS13}).

Coming back to Eq. \eqref{eq3} with general $m\geq1$, due to its
difficulty given by the play between the three exponents involved,
there are only a few works dealing with it. Suzuki \cite{Su02} gave
a detailed answer in the range $p>m$ to the question concerning
critical exponents limiting finite time blow up, both in the sense
of varying the reaction exponent $p$, but also the behavior of the
initial data $u_0(x)$ as $|x|\to\infty$. More precisely, he proved
the following results for \eqref{eq3}:

\medskip

\noindent (a) There exists an exponent
$p^*_{m,\sigma}:=m+(\sigma+2)/N$, such that if $m<p\leq
p^*_{m,\sigma}$, all nontrivial solutions to \eqref{eq3} blow up in
finite time (there is no global solution). This exponent
$p^*_{m,\sigma}$ is the analogous to the Fujita exponent for
\eqref{eq3}.

\medskip

\noindent (b) If $p>p^*_{m,\sigma}$, then there exists a constant
$A>0$ such that if the initial condition $u_0(x)$ satisfies
$$
\liminf\limits_{|x|\to\infty}|x|^{(\sigma+2)/(p-m)}u_0(x)>A,
$$
then the corresponding solution of the Cauchy problem blows up in
finite time.

\medskip

\noindent (c) If $p>p^*_{m,\sigma}$, then for any
$\a>(\sigma+2)/(p-m)$, there exists a constant $k>0$ such that if
for some $R>0$ sufficiently large,
$$
u_0(x)\leq k|x|^{-\a}, \quad |x|>R>0,
$$
then the corresponding solution to \eqref{eq1} with initial condition
$u_0$ exists globally in time.

The rate of convergence for a very general diffusion (known as
doubly nonlinear) was also investigated by Andreucci and Tedeev
\cite{AT05}. Restricting to \eqref{eq3}, they prove that, when
$m<p<m+2/N$ and $0<\sigma\leq N(p-m)/m$, any nonnegative solution to
\eqref{eq3} having the blow up time $T>0$ satisfies
$$
\|u(t)\|_{\infty}\leq K(T-t)^{-(\sigma+2)/[2(p-1)+\sigma(m-1)]},
\quad \frac{T}{2}<t<T.
$$
As we can see, all these results deal with the case $p>m$, letting
aside (due to technical reasons) the complementary case $1<m\leq p$.
Moreover, up to our knowledge, there is no work on the blow up sets
and blow up behavior for these cases, that is, addressing the third
and fourth questions in the list above.

More recently, the problem of blow up for non-homogeneous but
\emph{localized} reaction terms, that is, equations of the type
$$
u_t=\Delta u^m + a(x)u^p, \quad m>1, \ p>0,
$$
with $a(x)$ a compactly supported function (typically a
characteristic function of a bounded set) has been investigated,
starting from the work by Ferreira, de Pablo and V\'azquez
\cite{FdPV06} dealing with the one-dimensional case. The results
were then generalized to the equation posed in $\real^N$ in
\cite{KWZ11, Liang12} and also to the fast diffusion case $m<1$
\cite{BZZ11}. In all the above mentioned works, interesting
properties related to the Fujita-type exponent are proved: that is,
there are important and striking differences concerning the value of
the Fujita-type critical exponent that differs with respect to
dimension $N=1$, $N=2$ and $N\geq3$ and also in all these cases it
is different from the standard exponent of the homogeneous case.
Moreover, in \cite{FdPV06}, blow up rates, sets and profiles are
also established for the one-dimensional case. But all these works
rely deeply on the fact that the non-homogeneous reaction is
compactly supported (and in particular, there is no reaction close
to the spatial infinity).

This is why, our main goal is to study the influence of the
non-localized weight $|x|^{\sigma}$ (whose main property is that it
precisely weights more while approaching spatial infinity) on the
blow up set and behavior of solutions to \eqref{eq1} (and more
general \eqref{eq3}), trying to give some answers to these questions
that were up to now not properly studied. In the present work, we
restrict ourselves to dimension $N=1$ and the range of exponents
$1<m<p$, the complementary cases being left to be treated in further
papers due to important qualitative differences in the techniques
and results.

\bigskip

\noindent \textbf{Main results.} As it has been noticed since long,
special (particular) solutions, usually in self-similar form,
contain very important information concerning the qualitative
properties of solutions to \eqref{eq3}, and are likely to be
\emph{blow up profiles} for a large class of solutions, that is,
patterns to which solutions approach near their blow up time. That
is the reason for which we want to find and (if possible) classify
self-similar blow up solutions associated to \eqref{eq1}. These are
solutions to \eqref{eq1} (at least at a formal level) having the
particular form:
\begin{equation}\label{SSform}
u(x,t)=(T-t)^{-\alpha}f(\xi), \qquad \xi=|x|(T-t)^{\beta},
\end{equation}
for some positive exponents $\alpha$ and $\beta$ to be determined,
where $T\in(0,\infty)$ is the blow up time. Replacing the form given
in \eqref{SSform} into \eqref{eq1}, we find that the
\emph{self-similar profile} $f$ satisfies the following
non-autonomous differential equation
\begin{equation}\label{SSODE}
(f^m)''(\xi)-\alpha f(\xi)+\beta \xi f'(\xi)+\xi^{\sigma}f(\xi)^p=0, \qquad \xi\in[0,\infty)
\end{equation}
where
\begin{equation}\label{SSexp}
\alpha=\frac{\sigma+2}{2(p-1)+\sigma(m-1)}, \qquad
\beta=\frac{m-p}{2(p-1)+\sigma(m-1)}>0.
\end{equation}
We define our concept of solution we are looking for in the next
\begin{definition}\label{def1}
We say that $f$ solution to \eqref{SSODE} is a \textbf{good profile}
if it fulfills one of the following two properties related to its
behavior at $\xi=0$:

\indent (P1) $f(0)=a>0$, $f'(0)=0$.

\indent (P2) $f(0)=0$, $(f^m)'(0)=0$.

A good profile $f$ is called a \textbf{good profile with interface}
at some point $\eta\in(0,\infty)$ if
$$
f(\eta)=0, \qquad (f^m)'(\eta)=0, \qquad f>0 \ {\rm on} \
(\eta-\delta,\eta), \ {\rm for \ some \ } \delta>0.
$$
\end{definition}
With this definition, we can state our first main result.
\begin{theorem}[Existence of good profiles with interface]\label{th.exist}
For any $\sigma>0$, there exists at least one good profile with
interface $f$ to Eq. \eqref{SSODE}.
\end{theorem}
\noindent \textbf{Remark.} For $\sigma=0$, the analogous of Theorem
\ref{th.exist} is proved in \cite[Theorem 2, p. 187]{S4}. Let us
notice that for $\sigma=0$ there exist only good profiles fulfilling
condition (P1) above. The weighted reaction $|x|^{\sigma}u^p$ in
\eqref{eq1} introduces thus a sharp difference with respect to the
homogeneous case: profiles satisfying condition (P2) above may exist
and they have to be considered as good. Moreover, as we shall see,
profiles in (P2) present interesting properties with respect to the
blow up behavior.

\medskip

In view of the previous remark, a natural question arises: one can
ask in which conditions the good profiles with interface satisfy
condition (P1) and in which conditions they satisfy (P2). This is
the subject of the following two results which reflect a strong
influence of the magnitude of $\sigma>0$ on the blow up behavior.
\begin{theorem}[Good profiles with interface for $\sigma>0$ small]\label{th.small}
There exists $\sigma_*>0$ such that, for any
$\sigma\in(0,\sigma_*)$, any good profile with interface to Eq.
\eqref{SSODE} is of type (P1). In particular, the corresponding
solutions to Eq. \eqref{eq1} blow up globally.
\end{theorem}
In the following numerical experiment (see Figure \ref{fig1}), we represent the evolution of the profiles with interface for $\sigma$ sufficiently small, more precisely for $m=3$, $p=2$ and $\sigma=1$. Let us notice that profiles with interface at points $\eta$ large cut the axis $y=0$ at positive points, then the profiles with interface at small $\eta>0$ cut the axis $x=0$ at some positive point and are decreasing, and in the middle the good profiles with interface touch the vertical axis with varying slopes, which may be both positive and negative. In the middle there is one with $f(0)=a>0$ and slope $f'(0)=0$. All this is proved in Sections \ref{sec.exist} and \ref{sec.small}.
\begin{figure}[ht!]
  \begin{center}
  \includegraphics[width=10cm,height=7.5cm]{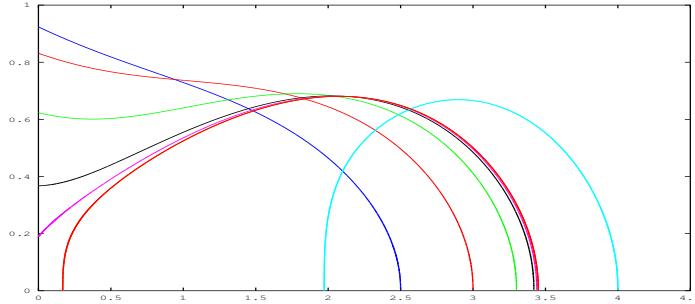}
  \end{center}
  \caption{Evolution of good profiles with interface for $\sigma$ sufficiently small} \label{fig1}
\end{figure}
This does not go very far from the homogeneous case, since the
result is similar for $\sigma=0$. But the next theorem produces a
sharp contrast to the homogeneous case.
\begin{theorem}[Good profiles with interface for $\sigma>0$ large]\label{th.large}
(a) For any $m>1$ and $p\in (1,m)$, there exists $\sigma^*=\sigma^*(m,p)>0$ depending on $m$ and $p$ such that, for $\sigma=\sigma^*$ there exists a good profile with interface solution to Eq. \eqref{SSODE} satisfying property (P2) in Definition \ref{def1} and moreover
\begin{equation}\label{case1}
f(\xi)\sim\left[\frac{m-1}{2m(m+1)}\right]^{1/(m-1)}\xi^{2/(m-1)},
\qquad {\rm as} \ \xi\to0.
\end{equation}
The corresponding solutions to Eq. \eqref{eq1} blows up globally in $\real$.

\medskip

(b) For any $m>1$ and $p\in(1,m)$ there also exists $\sigma_1\geq\sigma^*(m,p)$ sufficiently large such that for any $\sigma\in(\sigma_1,\infty)$, there exist good profiles with interface satisfying property (P2) in Definition \ref{def1} and moreover
\begin{equation}\label{case2}
f(\xi)\sim K\xi^{(\sigma+2)/(m-p)}, \ \ K>0, \qquad {\rm as} \
\xi\to0,
\end{equation}
and in this case, the corresponding solutions to Eq. \eqref{eq1} blow up at space infinity.
\end{theorem}
\noindent \textbf{Remark. Blow up sets.} In order to clarify the references about the blow up sets in the statement of Theorem \ref{th.large}, we define for a generic solution $u$ to \eqref{eq1} with initial condition $u_0(x):=u(x,0)$ and (finite) blow up time $T\in(0,\infty)$, the \emph{blow up set} \cite[Section 24]{QS} by
\begin{equation}\label{BUS}
B(u_0):=\{x\in\real: \exists(x_k,t_k)\in\real\times(0,T), \ t_k\to T, \ x_k\to x, \ {\rm and} \  |u(x_k,t_k)|\to\infty, \ {\rm as} \ k\to\infty\}.
\end{equation}
With this definition, we notice that for either a good profile with $f(0)=a>0$, that is, fulfilling assumption (P1) in Definition \ref{def1}, or for a profile behaving as in \eqref{case1} as $\xi\to0$, the corresponding solution $u$ blows up globally. Indeed, in the former, we have,
\begin{equation}\label{interm40}
u(x,t)=(T-t)^{-\alpha}f(|x|(T-t)^{\beta})\sim a(T-t)^{-\alpha}, \quad {\rm as} \ t\to T,
\end{equation}
while in the latter case, the solution $u$ satisfies
\begin{equation}\label{interm41}
u(x,t)\sim C(T-t)^{-\alpha+2\beta/(m-1)}|x|^{2/(m-1)}=C(T-t)^{-1/(m-1)}|x|^{2/(m-1)}, \quad {\rm as} \ t\to T,
\end{equation}
and in both cases blows up globally according to the definition of the blow up set \eqref{BUS}. An important remark is that, however, the {\bf blow up rate over fixed compact sets is different} in the two cases, as it readily follows from \eqref{interm40} and \eqref{interm41}. A sharper difference occurs for profiles behaving as in \eqref{case2} as $\xi\to0$. Indeed, for any $x\in\real$ fixed, we find
$$
u(x,t)\sim C(T-t)^{-\alpha+(\sigma+2)\beta/(m-p)}|x|^{(\sigma+2)/(m-p)}=C|x|^{(\sigma+2)/(m-p)}<\infty, \quad {\rm as} \ t\to T,
$$
hence these solutions remain bounded forever at any finite point. However, they still blow up at $t=T$, but only on curves $x(t)$ depending on $t$ such that $x(t)\to\infty$ as $t\to T$. This phenomenon is known in literature as \textbf{blow up at (space) infinity}, which seems to have been considered for the first time by Lacey \cite{La84}, and some other cases where it has been established (even for semilinear reaction-diffusion equation with sufficiently large initial data) appear in \cite{GU05, GU06}. In our opinion, this sharp difference with respect to the blow up set between solutions for $\sigma>0$ small and $\sigma>0$ large is one of the most interesting contributions of the present work.

In the following numerical experiment (see Figure \ref{fig2}), realized for $m=3$, $p=2$ and $\sigma=1,5$, one can see the evolution of the good profiles with interface as expressed in Theorem \ref{th.large}. The main difference with respect to $\sigma>0$ smaller appears to be that all the profiles with interface intersecting the vertical axis do that with negative slope, so that the good profile starts at $\xi=0$ with $f(0)=0$ and $(f^m)'(0)=0$.
\begin{figure}[ht!]
  \begin{center}
  \includegraphics[width=10cm,height=7.5cm]{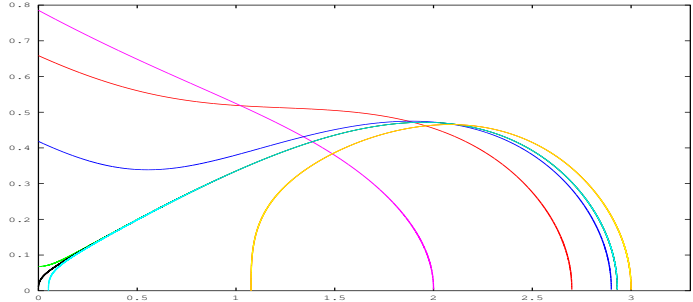}
  \end{center}
  \caption{Evolution of good profiles with interface for $\sigma$ sufficiently large}\label{fig2}
\end{figure}
Theorems \ref{th.small} and \ref{th.large} show that there is a very
strong \emph{influence of $\sigma$ on the blow up behavior} of
solutions to Eq. \eqref{eq1}, which we believe that is one of the
main points of interest of the paper. We thus show that the
weighted reaction introduces some unexpected differences with
respect to the homogeneous reaction, and it is logical that these
influences are noticed more when $\sigma>0$ increases, as the weight
becomes very strong at infinity.

Apart from the good profiles with interface, that are our main
object of interest throughout the paper, there exists another
category of good profiles, more precisely solutions $f$ to
\eqref{SSODE} decaying to zero as $\xi\to\infty$. We can also
classify them.
\begin{theorem}\label{th.decay}
There exists $\sigma_0>0$ such that for any $\sigma\in(0,\sigma_0)$
there exist good profiles satisfying property (P2) with both
behaviors \eqref{case1} and \eqref{case2} near $\xi=0$ and with the following
decay rate at space infinity
\begin{equation}\label{tail}
f(\xi)\sim \left(\frac{1}{p-1}\right)^{1/(p-1)}\xi^{-\sigma/(p-1)},
\qquad {\rm as} \ \xi\to\infty.
\end{equation}
Moreover, for any $\sigma>0$ there exists at least a good profile satisfying property (P2) with behavior \eqref{case2} near $\xi=0$ and with the decay rate at infinity given by \eqref{tail}.
\end{theorem}
\noindent \textbf{Remark.} The result in Theorem \ref{th.decay}
shows another difference between the cases $\sigma=0$ and $\sigma>0$, holding even for $\sigma>0$ small. Indeed, in the homogeneous case
$\sigma=0$, no solutions decaying to zero as $\xi\to\infty$ exist,
which readily follows as a byproduct of \cite[Theorem 2, p. 187]{S4}
and \cite[Lemma 3 and Corollary, p. 264-265]{S4}. While for
$\sigma>0$, there are even two types of such solutions, with sharply
different blow up behavior: global blow up for solutions as in
\eqref{case1} and blow up at space infinity for solutions as in
\eqref{case2}.

\medskip

Another interesting point in the paper is, in our opinion, the
\emph{general techniques we are using for the proofs}. Since already
three decades, the shooting method has imposed itself as a standard,
and very useful, technique in the classification of the self-similar
profiles to diffusion equations, in a variety of situations,
starting from blow up profiles for reaction-diffusion equations (see
for example \cite[Chapter 4]{S4}), to algebraic decay or extinction
profiles for diffusion equations with absorption (see as relevant
examples \cite{FV, CQW03}) or more recently, also for equations with
gradient terms (see for example \cite{IL13a, IL13b}). In all these
cited works, the goal was to classify profiles $f(\xi)$ solutions to
ordinary differential equations such that $f(0)=a>0$, $f'(0)=0$ and
$f$ decreasing while positive, and the shooting parameter is
$f(0)=a>0$.

In our case, this direct shooting technique cannot be used. First of all, there is the (non-trivial) technical problem that our
profiles are not decreasing; indeed, they might have one or more
local maximum points at $\xi>0$ due to the effect of the weight
$\xi^{\sigma}$. But even most important, we cannot apply the
standard shooting method since, as we shall see, we have to allow as
"good solutions", profiles \emph{starting from zero}, that is,
solutions $f(\xi)$ with $f(0)=0$ (and a suitable local behavior near
$\xi=0$) and at $\xi=0$, Eq. \eqref{SSODE} lacks the uniqueness
property in our range of exponents $1<p<m$. That is why, one cannot
use $f(0)$ as a shooting parameter. Still keeping the general idea
of the shooting, we instead construct our proofs on basis of a
\emph{backward shooting method from the interface point}. Indeed,
for any $\eta\in(0,\infty)$, there exists a unique profile $f$ such
that $f(\eta)=0$, $f>0$ in $(0,\eta)$ and $(f^m)'(\eta)=0$, and we
will use this $\eta\in(0,\infty)$ as the shooting parameter and try
to trace backward the unique profile $f_{\eta}$ vanishing at the
given $\eta$. We think that this particular form of applying the shooting technique was very seldom exploited in literature, as we only have knowledge of the paper \cite{GP76} using it in this way. Let us stress at this point the other big technique we use in the present paper, that of a careful and complete study of a phase space associated to a system of three ODEs. Such a technique is rather involved and not very common for systems of more than two ODEs, see for example \cite{BHK01, dPS00} or our companion paper \cite{IS1} where such analysis are also performed.

\section{Self-similar profiles. The phase space}\label{sec.local}

This rather technical section is devoted to the local analysis of a
suitable phase space associated to a quadratic autonomous dynamical
system into which \eqref{SSODE} can be transformed. All the results
in this section can be seen as technical preliminaries for the
global analysis which is performed later. We thus transform Eq.
\eqref{SSODE} into an autonomous dynamical system of three equations
of order one, by letting in a first step
\begin{equation}\label{PSchange1}
x(\eta)=f^{m-1}(\xi), \ \ y(\eta)=(f^{m-2}f')(\xi), \ \ z(\eta)=\xi,
\end{equation}
and the new independent variable $\eta=\eta(\xi)$ satisfies
$$
\frac{d\eta}{d\xi}=mf^{m-1}(\xi)=mx(\eta).
$$
With this notation, it is easy to check that Eq. \eqref{SSODE}
transforms into the following system:
\begin{equation}\label{PSsyst1}
\left\{\begin{array}{ll}\dot{x}=m(m-1)xy,\\
\dot{y}=-my^2-\beta yz+\alpha x-|z|^{\sigma}x^{(m+p-2)/(m-1)},\\
\dot{z}=mx.\end{array}\right.
\end{equation}
This system, however, have an important disadvantage: in it, there
are still terms with fractional powers, which sometimes are
difficult to handle. That is why we also consider a different change of variables leading to an autonomous system which is quadratic. More
precisely, let
\begin{equation}\label{PSchange2}
X(\eta)=\xi^{-2}f^{m-1}(\xi), \ \
Y(\eta)=\xi^{-1}f^{m-2}(\xi)f'(\xi), \ \
Z(\eta)=\xi^{\sigma}f^{p-1}(\xi),
\end{equation}
where the new independent variable $\eta=\eta(\xi)$ satisfies
$$
\frac{d\eta}{d\xi}=\frac{1}{m}\xi f^{1-m}(\xi).
$$
With this notation, Eq. \eqref{SSODE} transforms into the following
system (the calculations leading to it are sometimes tedious, but
straightforward and are left to the reader):
\begin{equation}\label{PSsyst2}
\left\{\begin{array}{ll}\dot{X}=mX[(m-1)Y-2X],\\
\dot{Y}=-mY^2-\beta Y+\alpha X-mXY-XZ,\\
\dot{Z}=mZ[(p-1)Y+\sigma X].\end{array}\right.
\end{equation}
We will use freely both systems in the subsequent analysis, as both
of them have advantages and disadvantages for the study. We keep the notation with capital letters $X$, $Y$, $Z$ when referring
to the system \eqref{PSsyst2} and with lower case letters $x$, $y$,
$z$ when referring to the system \eqref{PSsyst1}. Notice that in
both systems, the planes $\{X=0\}$ and $\{Z=0\}$ (respectively $\{x=0\}$)) are invariant and we work only with $X\geq0$ and $Z\geq0$ (respectively $x\geq0$,
$z\geq0$), only $Y$ (resp. $y$) may change sign.

\medskip

\noindent \textbf{Local analysis of the finite critical points in
the phase space.} We will take as phase space of reference, the one
associated to the system \eqref{PSsyst2}. We readily notice that
this system has four categories of finite critical points: three
isolated ones
$$
P_0=(0,0,0), \ \ P_1=\left(0,-\frac{\beta}{m},0\right), \ \
P_2=\left(\frac{m-1}{2m(m+1)},\frac{1}{m(m+1)},0\right),
$$
and a line of critical points on the $OZ$-axis, denoted by
$P_{\gamma}=(0,0,\gamma)$, for any $\gamma>0$. Of course, one can
say that $P_0$ belongs to the same axis (with $\gamma=0$), but as we
shall see, the origin is qualitatively very different from all the
other $P_{\gamma}$ with $\gamma>0$ and should be considered
separately.

\medskip

\begin{lemma}[Analysis of the point $P_0=(0,0,0)$]\label{lem.1}
The system in a neighborhood of the critical point $P_0$ has a
one-dimensional stable manifold and a two-dimensional center
manifold. The connections over the center manifold go out of the
point $P_0$ and contain profiles with the local behavior:
\begin{equation}\label{Beh02}
f(\xi)\sim k\xi^{(\sigma+2)/(m-p)}=k\xi^{\alpha/\beta}, \qquad
f(0)=0,
\end{equation}
for any constant $k>0$.
\end{lemma}
\begin{proof}
The linearization of the system \eqref{PSsyst2} near this critical
point has the matrix:
$$
M(P_0)=\left(
      \begin{array}{ccc}
        0 & 0 & 0 \\
        \alpha & -\beta & 0 \\
        0 & 0 & 0 \\
      \end{array}
    \right)
$$
hence it has a one-dimensional stable manifold (corresponding to the
eigenvalue $-\beta$) and a two-dimensional center manifold. As we
are interested in the orbits going out of $P_0$ in the phase space
(if any), we will analyze this center manifold and the flow on it
following the recipe given in \cite[Theorem 1, Section 2.12]{Pe}. In
order to put the system in a "canonical form" near $P_0$, we
introduce the new variable
$$
W:=\beta Y-\alpha X
$$
and after straightforward calculations, we obtain the new system
\begin{equation}\label{interm1}
\left\{\begin{array}{ll}\dot{X}=\frac{m(m-1)}{\beta}XW+\frac{m}{\beta}X^2,\\
\dot{W}=-\beta W-\frac{m}{\beta}W^2-\frac{m[\alpha(m+1)+\beta]}{\beta}XW-\frac{m\alpha}{\beta}(\alpha+\beta+1)X^2-\beta XZ,\\
\dot{Z}=\left[\frac{m(p-1)\alpha}{\beta}+m\sigma\right]XZ+\frac{m(p-1)}{\beta}WZ.\end{array}\right.
\end{equation}
We are then in position to apply Theorem 1 in \cite[Section
2.12]{Pe}, and to look for a center manifold of the form
$$
W=h(X,Z):=aX^2+bXZ+cZ^2+O(|(X,Z)|^3).
$$
We introduce this ansatz in Theorem 1 in \cite[Section 2.12]{Pe} and conclude that the center manifold is given by the equation
$$
h(X,Z)=-\frac{m\alpha(\alpha+\beta+1)}{\beta^2}X^2-XZ+O(|(X,Z)|^3).
$$
Using the same theorem, replacing $h(X,Z)$ by its formula and taking into account the expressions of $\alpha$ and $\beta$ from \eqref{SSexp}, we
get that the flow on the center manifold in a neighborhood of the point $P_0$ is given by the following system:
$$
\left\{\begin{array}{ll}\dot{X}=\frac{m}{\beta}X^2+O(|(X,Z)|^3),\\
\dot{Z}=\frac{m}{\beta}XZ+O(|(X,Z)|^3),\end{array}\right.
$$
Integrating this system up to first order, we find that $Z\sim kX$ for some positive constant $k$. Coming back to profiles in \eqref{PSchange2}, we
obtain
$$
\xi^{\sigma}f^{p-1}(\xi)\sim k\xi^{-2}f^{m-1}(\xi), \qquad k>0,
$$
whence the orbits going out of $P_0$ on the two-dimensional center manifold contain profiles of the form given in \eqref{Beh02}.
\end{proof}
\begin{lemma}[Analysis of the point
$P_1=\left(0,-\beta/m,0\right)$]\label{lem.2} The system in a
neighborhood of the critical point $P_1$ has a one-dimensional
unstable manifold and a two-dimensional stable manifold. The orbits
entering $P_1$ on the stable manifold contain profiles such that
\begin{equation}\label{IFeq}
f(\xi)\sim\left[C-\frac{\beta(m-1)}{2m}\xi^2\right]^{1/(m-1)}, \quad
C>0,
\end{equation}
for $\xi\to\xi_0=\sqrt{2mC/(m-1)\beta}\in(0,\infty)$.
\end{lemma}
\begin{proof}
The linearization of the system \eqref{PSsyst2} near this critical
point has the matrix:
$$M(P_1)=\left(
  \begin{array}{ccc}
    -\beta(m-1) & 0 & 0 \\
    \alpha+\beta & \beta & 0 \\
    0 & 0 & -(p-1)\beta \\
  \end{array}
\right)
$$
with eigenvalues $\lambda_1=-\beta(m-1)$, $\lambda_2=\beta$ and
$\lambda_3=-(p-1)\beta$ and respective eigenvectors (not normalized)
$e_1=(1,-(\alpha+\beta)/m\beta,0)$, $e_2=(0,1,0)$ and $e_3=(0,0,1)$.
Then, there is a two-dimensional stable manifold, with orbits
entering the point $P_1$ in the phase space, and (as it is easy to
check) a unique orbit going out of $P_1$ along the $Y$-axis. We will
be interested in the profiles contained in the orbits entering
$P_1$. Taking into account the change of variables
\eqref{PSchange2}, in particular that
\begin{equation}\label{interm2}
Y(\xi)=\xi^{-1}(f^{m-2}f')(\xi)=\frac{\xi^{-1}}{m-1}(f^{m-1})'(\xi)
\end{equation}
and that $Y(\xi)\sim -\beta/m$ when entering $P_1$ either when
$\xi\to\infty$ or when $\xi\to\xi_0\in(0,\infty)$, by integration in
\eqref{interm2} we find that the orbits entering $P_1$ contain
profiles satisfying \eqref{IFeq}. We readily notice that these
profiles satisfy the flow equation at the zero point $\xi=\xi_0$,
that is
$$
\lim\limits_{\xi\to\xi_0}(f^{m})'(\xi)=0,
$$
hence these profiles satisfying \eqref{IFeq} present an interface point at finite distance
$\xi=\xi_0\in(0,\infty)$. These will be the profiles we are mostly
interested in within the present paper, as their existence is
characteristic to the range of exponents $1<p<m$ even in the
homogeneous case $\sigma=0$ \cite{S4}.
\end{proof}
As an interesting remark, it will be useful for some purposes to see
these orbits also in the first phase space, that is, the one
associated to the system \eqref{PSsyst1}. In that system, and taking
into account \eqref{IFeq}, we notice that
$$
my(\xi)+\beta z(\xi)=\frac{m}{m-1}(f^{m-1})'(\xi)+\beta\xi\to0,
\quad {\rm as} \ \xi\to\xi_0,
$$
so that, viewed in the phase space associated to \eqref{PSsyst1},
the critical point $P_1$ expands into the critical half-line
$my+\beta z=0$ with $z>0$ and $y<0$.

\begin{lemma}[Analysis of the point
$P_2=((m-1)/2m(m+1),1/m(m+1),0)$]\label{lem.3} The system in a
neighborhood of the critical point $P_2$ has a two-dimensional
stable manifold and a one-dimensional unstable manifold. The stable
manifold is included in the invariant plane $Z=0$. There exists a
unique orbit going out of $P_2$, containing profiles which locally
satisfy
\begin{equation}\label{Beh01}
f(0)=0, \quad
f(\xi)\sim\left[\frac{m-1}{2m(m+1)}\right]^{1/(m-1)}\xi^{2/(m-1)}-\psi(\sigma)\xi^{(\sigma+2)/(m-p)}, \
\ {\rm as} \ \xi\to0,
\end{equation}
where $\psi(\sigma)$ is a coefficient depending on $\sigma$ such that
\begin{equation}\label{Beh01bis}
\lim\limits_{\sigma\to\infty}\psi(\sigma)=0.
\end{equation}
\end{lemma}
\begin{proof}
The linearization of the system \eqref{PSsyst2} near this critical
point has the matrix:
$$
M(P_2)=\left(
  \begin{array}{ccc}
    -\frac{m-1}{m+1} & \frac{(m-1)^2}{2(m+1)} & 0 \\
    \alpha-\frac{1}{m+1} & -\beta-\frac{m+3}{2(m+1)} & -\frac{m-1}{2m(m+1)} \\
    0 & 0 & \frac{2(p-1)+\sigma(m-1)}{2(m+1)} \\
  \end{array}
\right)
$$
having eigenvalues $\lambda_1$, $\lambda_2$ and
$$\lambda_3=\frac{2(p-1)+\sigma(m-1)}{2(m+1)}.$$ Moreover, it is
immediate to see that
$$
\lambda_1+\lambda_2=-\frac{m-1}{m+1}-\beta-\frac{m+3}{2(m+1)}<0
$$
and
\begin{equation*}
\begin{split}
\lambda_1\lambda_2&=\frac{m-1}{m+1}\left[\beta+\frac{m+3}{2(m+1)}\right]-\frac{(m-1)^2}{2(m+1)}\left[\alpha-\frac{1}{m+1}\right]\\
&=\frac{m-1}{2(m+1)}>0,
\end{split}
\end{equation*}
so that $\lambda_1<0$ and $\lambda_2<0$. Thus, there is a
two-dimensional stable manifold, composed by orbits entering $P_2$
and included in the invariant plane $Z=0$, and it is easy to check
that there exists only one orbit (for any $\sigma>0$ fixed) going
out of $P_2$ tangent to the eigenvector associated to the eigenvalue
$\lambda_3$. In the sequel, we will be interested in this unique
orbit. Let us notice first that this orbit contains profiles for
which
\begin{equation}\label{term1}
\lim\limits_{\xi\to0}\xi^{-2}f^{m-1}(\xi)=\frac{m-1}{2m(m+1)}.
\end{equation}
Moreover, elementary but rather tedious calculations show that the eigenvector of the matrix $M(P_2)$ corresponding to the eigenvalue $\lambda_3$ is $(x(\sigma),y,z(\sigma))$ with
\begin{eqnarray*}
&x(\sigma)=-\frac{(m-1)^2}{2(m+p-2)+\sigma(m-1)}, \quad y=-1,\\
&z(\sigma)=\frac{2m}{m-1}\left[-\frac{(\alpha(m+1)-1)(m-1)^2}{2(m+p-2)+\sigma(m-1)}+\frac{2(m+1)\beta+m+2p+1+\sigma(m-1)}{2}\right],
\end{eqnarray*}
where the choice of the signs was taken in order for $z(\sigma)>0$. In particular we infer that the components $X$ and $Y$ decrease very close to $P_2$ along the orbit going out of $P_2$. Defining $\psi(\sigma):=|x(\sigma)/z(\sigma)|^{1/(m-p)}$, we readily infer \eqref{Beh01bis}. Moreover, the second (algebraic) order in the local behavior of a profile $f$ near $P_2$ is given by the relation $$X-\frac{m-1}{2m(m+1)}\sim|\psi(\sigma)|^{m-p}Z$$ as $\xi\to0$, which taking into account the definitions of $X$ and $Z$ in \eqref{PSchange2} and the first term of the expansion given by \eqref{term1}, readily gives \eqref{Beh01}.
\end{proof}
\begin{lemma}[Analysis of the points $P_{\gamma}=(0,0,\gamma)$, for
$\gamma>0$]\label{lem.4} There exists a unique
\begin{equation}\label{eq.gamma}
\gamma_0=\alpha+\frac{\beta\sigma}{p-1}=\frac{1}{p-1}
\end{equation}
such that the critical point $P_{\gamma_0}$ is an attractor. The
orbits entering $P_{\gamma_0}$ contain profiles which decay at
infinity with the rate
\begin{equation}\label{queue}
f(\xi)\sim\left(\frac{1}{p-1}\right)^{1/(p-1)}\xi^{-\sigma/(p-1)},
\qquad {\rm as} \ \xi\to\infty.
\end{equation}
For $\gamma\in(0,\infty)$ with $\gamma\neq\gamma_0$ there is no orbit entering $P_{\gamma}$.
\end{lemma}
\begin{proof}
The linearization of the system \eqref{PSsyst2} near this critical
point has the matrix:
$$
M(P_{\gamma})=\left(
  \begin{array}{ccc}
    0 & 0 & 0 \\
    \alpha-\gamma & -\beta & 0 \\
    m\sigma\gamma & m(p-1)\gamma & 0 \\
  \end{array}
\right)
$$
having a double eigenvalue $\lambda_1=\lambda_3=0$ and
$\lambda_2=-\beta$. That is, these points have a two-dimensional
center manifold and a one-dimensional stable manifold. The most
involved part is the analysis of the center manifold. Following once
more the recipe given in \cite[Section 2.12]{Pe}, we first change
the system in order to transfer the critical point $P_{\gamma}$ to
the origin by letting $Z=\bar{Z}+\gamma$. The new system writes:
\begin{equation}\label{interm3}
\left\{\begin{array}{ll}\dot{X}=mX[(m-1)Y-2X],\\
\dot{Y}=-mY^2-\beta Y+(\alpha-\gamma)X-mXY-X\bar{Z},\\
\dot{\bar{Z}}=m\bar{Z}[(p-1)Y+\sigma X]+m\sigma\gamma X+m(p-1)\gamma
Y.\end{array}\right.
\end{equation}
We now do a second change of variable by letting
$$
\bar{Z}=H-kY,
$$
for some $k$ to be determined later, in order to eliminate the
linear term in $Y$ from the third equation in the system
\eqref{interm3}. We thus get
\begin{equation*}
\begin{split}
\dot{H}&=\dot{\bar{Z}}+k\dot{Y}=m(H-kY)[(p-1)Y+\sigma
X]+m\sigma\gamma X+m(p-1)\gamma Y\\&-mkY^2-k\beta
Y+k(\alpha-\gamma)X-mkXY-kX(H-kY)\\&=[m\sigma\gamma+k(\alpha-\gamma)]X+[m(p-1)Y+(m\sigma-k)X]H-mkpY^2-k(m\sigma+m-k)XY,
\end{split}
\end{equation*}
provided
$$
k:=\frac{m(p-1)\gamma}{\beta}.
$$
With this notation, we finally do the last change of variable needed
in order to analyze the center manifold near $P_{\gamma}$, by
replacing $Y$ by the new variable
$$
G:=\beta Y-(\alpha-\gamma)X.
$$
We then calculate the new system by doing:
\begin{equation}\label{interm4}
\begin{split}
\dot{G}&=\beta\dot{Y}-(\alpha-\gamma)\dot{X}=-\beta
G-\frac{m}{\beta}G^2-\left[\frac{m(m+1)(\alpha-\gamma)}{\beta}+(m-k)\right]GX\\
&-(\alpha-\gamma)\left[\frac{m^2(\alpha-\gamma)}{\beta}-(m+k)\right]X^2-\beta
XH.
\end{split}
\end{equation}
The equation for $X$ writes in the new variables $X, G, H$:
\begin{equation}\label{interm5}
\dot{X}=\frac{m(m-1)}{\beta}XG+\left[\frac{m(m-1)(\alpha-\gamma)}{\beta}-2m\right]X^2,
\end{equation}
and finally, the equation satisfied by $H$ writes
\begin{equation}\label{interm6}
\dot{H}=[m\sigma\gamma+k(\alpha-\gamma)]X+D_1XH-D_2X^2+{\rm terms \ containing \ G},
\end{equation}
with
$$
D_1=\left[\frac{m(\alpha-\gamma)(p-1)}{\beta}+m\sigma-k\right], \ D_2=\frac{k(\alpha-\gamma)}{\beta}\left[\frac{mp(\alpha-\gamma)}{\beta}+m\sigma+m-k\right]
$$
where we stress that the terms containing $G$ are at least of quadratic order. We are now ready to apply Theorem 1 in
\cite[Section 2.12]{Pe} to the system formed by the equations (in
this order) \eqref{interm5}, \eqref{interm4} and \eqref{interm6}. To
this end, we look for a center manifold of the form
$$
G(X,H)=aX^2+bXH+cH^2+O(|(X,H)|^3),
$$
with coefficients $a$, $b$ and $c$ to be determined. After rather long calculations, the equation of the center manifold near the points
$P_{\gamma}$ becomes
\begin{equation}\label{interm7}
G(X,H)=DX^2-XH+O(|(X,H)|^3),
\end{equation}
where
$$
D:=-\frac{1}{\beta^2}\left[-2(\alpha-\gamma)\beta k-m\beta(\sigma\gamma+\alpha-\sigma)+m^2(\alpha-\gamma)^2\right].
$$
We now replace $G(X,H)$ from \eqref{interm7} into the equations
\eqref{interm5} and \eqref{interm6} and we infer that locally near the point $P_{\gamma}$ the flow on the
center manifold is given by the reduced system (where the particular form of the higher order terms comes from \eqref{interm5}, \eqref{interm6}, \eqref{interm7} and an easy proof by induction)
\begin{equation}\label{flowgamma}
\left\{\begin{array}{ll}\dot{X}=\left[\frac{m(m-1)(\alpha-\gamma)}{\beta}-2m\right]X^2+\frac{m(m-1)D}{\beta}X^3-\frac{m(m-1)}{\beta}X^2H+XO(|(X,H)|^3)\\
\dot{H}=\left[m\sigma\gamma+k(\alpha-\gamma)\right]X+D_1XH+D_2X^2+XO(|(X,H)|^2).\end{array}\right.
\end{equation}
For $\gamma=\gamma_0:=1/(p-1)$ the linear term in the second equation of \eqref{flowgamma} vanishes. Noticing that
$$
L:=\frac{m(m-1)(\alpha-\gamma_0)}{\beta}-2m=-m\frac{\sigma(m-1)+2(p-1)}{p-1}<0, \ D_1=-\frac{m}{\beta}<0,
$$
by changing the independent variable as
\begin{equation}\label{change}
d\theta=X d\eta,
\end{equation}
one can reduce the order of the system \eqref{flowgamma} dividing by $X$ in the right-hand side in the region $\{X>0\}$ and thus get a nonlinear system for which the linearization near the point $(0,0)$ has eigenvalues $L$ and $D_1$, both negative. This, together with the fact that in the matrix $M(P_{\gamma_0})$ we have $\lambda_2=-\beta<0$, prove that $P_{\gamma_0}$ is an attractor. By performing the same change of variable \eqref{change} for $\gamma\neq 1/(p-1)$, that is, $m\sigma\gamma+k(\alpha-\gamma)\neq0$, in the topologically equivalent system obtained $(0,0)$ is no longer a critical point, thus it is easy to verify that there is no connection from the phase plane into the point $P_{\gamma}$ for $\gamma\neq\gamma_0=1/(p-1)$. Coming back to profiles and noticing that $Z=1/(p-1)$ for all orbits entering $P_{\gamma_0}$, we find that all the orbits entering this attractor contain profiles supported in the whole $\real$ and decaying at infinity with the rate given in \eqref{queue}.
\end{proof}
These profiles might give us very interesting solutions to
\eqref{eq1} with this typical decay. In particular, for the
homogeneous case $\sigma=0$, this attractor corresponds to the
constant profile
$$
f\equiv\left(\frac{1}{p-1}\right)^{1/(p-1)},
$$
but let us notice that for $\sigma>0$, these profiles decay to 0 as
$\xi\to\infty$ and they will be considered in the sequel.

\medskip

\noindent \textbf{Local analysis of the critical points at space
infinity.} Apart from the finite critical points, there are several critical points at infinity for the system
\eqref{PSsyst2}. In order to study the critical points at infinity,
we first pass to the Poincar\'e hypersphere following the recipes
given for example in \cite[Section 3.10]{Pe}. In order to pass to
the Poincar\'e hypersphere, we pass to new variables
$(\overline{X},\overline{Y},\overline{Z},W)$ by letting:
$$
X=\frac{\overline{X}}{W}, \ \ Y=\frac{\overline{Y}}{W}, \ \
Z=\frac{\overline{Z}}{W}
$$
and according to \cite[Theorem 4, Section 3.10]{Pe}, the critical
points at infinity in the phase space associated to the system
\eqref{PSsyst2} lie on the Poincar\'e hypersphere at points
$(\overline{X},\overline{Y},\overline{Z},0)$ where
$\overline{X}^2+\overline{Y}^2+\overline{Z}^2=1$ and the following
system is fulfilled:
\begin{equation}\label{Poincare1}
\left\{\begin{array}{ll}\overline{X}Q_2(\overline{X},\overline{Y},\overline{Z})-\overline{Y}P_2(\overline{X},\overline{Y},\overline{Z})=0,\\
\overline{X}R_2(\overline{X},\overline{Y},\overline{Z})-\overline{Z}P_2(\overline{X},\overline{Y},\overline{Z})=0,\\
\overline{Y}R_2(\overline{X},\overline{Y},\overline{Z})-\overline{Z}Q_2(\overline{X},\overline{Y},\overline{Z})=0,\end{array}\right.
\end{equation}
where $P_2$, $Q_2$ and $R_2$ are the homogeneous second degree parts
of the polynomials in the right hand side of the system \eqref{PSsyst2},
that is
\begin{equation*}
\begin{split}
&P_2(\overline{X},\overline{Y},\overline{Z})=m\overline{X}((m-1)\overline{Y}-2\overline{X}),\\
&Q_2(\overline{X},\overline{Y},\overline{Z})=-m\overline{Y}^2-m\overline{X}\overline{Y}-\overline{X}\overline{Z},\\
&R_2(\overline{X},\overline{Y},\overline{Z})=m\overline{Z}((p-1)\overline{Y}+\sigma\overline{X})
\end{split}
\end{equation*}
The system \eqref{Poincare1} thus becomes
\begin{equation}\label{Poincare2}
\left\{\begin{array}{ll}\overline{X}(-m^2\overline{Y}^2+m\overline{X}\overline{Y}-\overline{X}\overline{Z})=0,\\
m\overline{X}\overline{Z}((\sigma+2)\overline{X}-(m-p)\overline{Y})=0,\\
\overline{Z}(mp\overline{Y}^2+m(\sigma+1)\overline{X}\overline{Y}+\overline{X}\overline{Z})=0,\end{array}\right.
\end{equation}
Taking into account that we are considering only points with
coordinates $\overline{X}\geq0$ and $\overline{Z}\geq0$, we find the
following five critical points at infinity (on the Poincar\'e
hypersphere):
$$
Q_1=(1,0,0,0), \ \ Q_{2,3}=(0,\pm1,0,0), \ \ Q_4=(0,0,1,0), \ \
Q_5=\left(\frac{m}{\sqrt{1+m^2}},\frac{1}{\sqrt{1+m^2}},0,0\right).
$$
We perform below the local analysis of these points following the recipes given in
\cite[Theorem 5, Section 3.10]{Pe}.
\begin{lemma}\label{lem.inf1}
The critical point at infinity represented as $Q_1=(1,0,0,0)$ in the
Poincar\'e hypersphere is an unstable node. The orbits going out of
this point to the finite part of the phase space contain profiles
$f(\xi)$ such that $f(0)=a>0$ with any possible behavior of the
derivative $f'(0)$.
\end{lemma}
\begin{proof}
Applying part (a) of \cite[Theorem 5, Section 3.10]{Pe}, we obtain
that the flow of the system in a neighborhood of the point $Q_1$ is
topologically equivalent to the flow defined (in a neighborhood of
the origin) by the system:
\begin{equation}\label{systinf1}
\left\{\begin{array}{ll}-\dot{y}=-my+z-\alpha w+m^2y^2+\beta yw,\\
-\dot{z}=-m(\sigma+2)z+m(m-p)yz,\\
-\dot{w}=-2mw+m(m-1)yw,\end{array}\right.
\end{equation}
where the minus sign has been chosen according to the direction of
the flow in the original system \eqref{PSsyst2}. This follows for example from the
first equation in \eqref{PSsyst2}, which writes
$$
\dot{X}=mX[(m-1)Y-2X],
$$
and as $X/Y\to\infty$ when approaching $Q_1$, it follows that $\dot{X}<0$ in a neighborhood of the point $Q_1$, which gives the minus sign above. Thus, the linearization of the system \eqref{systinf1} in a neighborhood of the origin has the matrix
$$
M(Q_1)=\left(
         \begin{array}{ccc}
           m & -1 & \alpha \\
           0 & m(\sigma+2) & 0 \\
           0 & 0 & 2m \\
         \end{array}
       \right),
$$
showing that $Q_1$ is an unstable node. In order to classify the
profiles contained in the orbits going out of $Q_1$, we notice that,
in the variables of the equivalent system \eqref{systinf1},
$$
\frac{dz}{dw}\sim \frac{\sigma+2}{2}\frac{z}{w},
$$
whence by direct integration, $z\sim Cw^{(\sigma+2)/2}$ for some
$C\in\real$. Coming back to the original variables in the system
\eqref{PSsyst2} and recalling that the projection of the Poincar\'e
hypersphere in order to arrive to \eqref{systinf1} has been done by
dividing with the $X$ variable (as explained in \cite[Section
3.10]{Pe}), we infer
$$
\frac{Z}{X}\sim K\frac{1}{X^{(\sigma+2)/2}},
$$
whence by direct substitution in \eqref{PSchange2} we obtain
$$
f(\xi)\sim K>0, \quad {\rm as} \ \xi\to0,
$$
as stated. The fact that we have to choose $K>0$ above is due to the
fact that the profiles resulting from taking $K=0$ are already
classified in Lemma \ref{lem.3} and they go out of the finite point
$P_2$.
\end{proof}
\noindent \textbf{Remark.} We can delve deeper into the local
behavior as $\xi\to0$ of the profiles going out of $Q_1$. Since
$z\sim Cw^{(\sigma+2)/2}$, we readily get that
$$
\frac{dy}{dw}\sim\frac{1}{2}\frac{y}{w}+\frac{\alpha}{2m},
$$
whence by substitution in $X$, $Y$, $Z$ variables we get
\begin{equation}\label{intermY}
Y\sim KX^{1/2}+\frac{\alpha}{m}, \qquad K\in\real
\end{equation}
Thus, when the constant $K\neq0$, we have profiles with $Y\sim
KX^{1/2}$, thus
$f(\xi)\sim(C+K\xi)^{2/(m-1)}$ as $\xi\to0$, with non-zero
derivative at $\xi=0$. On the other hand, when $K=0$ in
\eqref{intermY}, we get that $Y\sim\alpha/m$, whence by direct
integration
$$
f(\xi)\sim\left(K+\frac{\alpha(m-1)}{2m}\xi^2\right)^{1/(m-1)},
\qquad K>0,
$$
thus obtaining here the profiles with initial conditions
$f(0)=a=K^{1/(m-1)}>0$, $f'(0)=0$ that are interesting for our
goals.
\begin{lemma}\label{lem.inf2}
The critical points at infinity represented as
$Q_{2,3}=(0,\pm1,0,0)$ in the Poincar\'e hypersphere are an unstable
node, respectively a stable node. The orbits going out of $Q_2$ to
the finite part of the phase space contain profiles $f(\xi)$ such
that there exists $\xi_0\in(0,\infty)$ with $f(\xi_0)=0$,
$f'(\xi_0)>0$. The orbits entering the point $Q_3$ and coming from
the finite part of the phase space contain profiles $f(\xi)$ such
that there exists $\xi_0\in(0,\infty)$ with $f(\xi_0)=0$,
$f'(\xi_0)<0$.
\end{lemma}
By convention, we will say that all these profiles are of
\emph{changing sign} type.
\begin{proof}
Applying part (b) of \cite[Theorem 5, Section 3.10]{Pe}, we obtain
that the flow of the system in a neighborhood of the points
$Q_{2,3}$ is topologically equivalent to the flow defined (in a
neighborhood of the origin) by the system:
\begin{equation}\label{systinf2}
\left\{\begin{array}{ll}\pm\dot{x}=-m^2x-\beta xw+mx^2+\alpha x^2w-x^2z,\\
\pm\dot{z}=-mpz-m(\sigma+1)xz-\beta zw+\alpha xzw-xz^2,\\
\pm\dot{w}=-mw-\beta w^2-mxw+\alpha xw^2-xzw.\end{array}\right.
\end{equation}
Since approaching the points $Q_{2,3}$ we have $|Y/X|\to\infty$ and $|Y/Z|\to\infty$, it follows from the second equation of the original system \eqref{PSsyst2}, that is,
$$
\dot{Y}=-mY^2-\beta Y-mXY+\alpha X-XZ,
$$
that in a neighborhood of the points $Q_2$ and $Q_3$, $\dot{Y}<0$, whence $Y$ is decreasing. This gives the direction of the flow near these points and shows that in the previous system \eqref{systinf2} we have to choose the minus sign for the point $Q_2=(0,1,0,0)$ and the plus sign for $Q_3=(0,-1,0,0)$. The matrix of the linearization of the system \eqref{systinf2} in a neighborhood of the origin is
$$
M(Q_{2,3})=\left(
             \begin{array}{ccc}
               -m^2 & 0 & 0 \\
               0 & -mp & 0 \\
               0 & 0 & -m \\
             \end{array}
           \right),
$$
hence by the previous choice of signs, $Q_2$ is an unstable node and
$Q_3$ is a stable node. In order to classify the profiles contained
in the orbits going out of $Q_2$ (respectively entering $Q_3$), we
notice that, in the variables of the equivalent system
\eqref{systinf2},
$$
\frac{dx}{dw}\sim m\frac{x}{w},
$$
whence by direct integration, $x\sim Cw^m$ for some $C>0$. Coming
back to the original variables in the system \eqref{PSsyst2} and
recalling that the projection of the Poincar\'e hypersphere in order
to arrive to \eqref{systinf2} has been done by dividing with the $Y$
variable (as explained in \cite[Section 3.10]{Pe}), we infer
$$
\frac{X}{Y}\sim C\frac{1}{Y^m},
$$
and by direct integration we obtain
\begin{equation}\label{interm13}
f(\xi)\sim\left[K+C\xi^{2m/(m-1)}\right]^{1/m}, \quad K, \
C\in\real.
\end{equation}
Let us notice that for the orbits entering $Q_3$, $Y(\xi)<0$ in a
neighborhood of it, which means that the profiles contained in such
orbits have $f'<0$ and $C<0$, thus compulsory $K>0$ in the formula
\eqref{interm13}. Then there exists $\xi_0\in(0,\infty)$ such that
$f(\xi_0)=0$, $f'(\xi_0)<0$. On the other hand, the profiles going
out of $Q_2$ do it with $Y>0$ (although decreasing), hence $C>0$ in
\eqref{interm13}. Thus, for part of these profiles, more precisely
those with $K<0$ in \eqref{interm13}, there exists
$\xi_0\in(0,\infty)$ such that $f(\xi_0)=0$, $f'(\xi_0)>0$, as
stated.
\end{proof}
We next analyze locally the flow in the phase space system near
$Q_5$.
\begin{lemma}\label{lem.inf3}
The critical point at infinity represented as
$Q_5=(m/\sqrt{1+m^2},1/\sqrt{1+m^2},0,0)$ in the Poincar\'e
hypersphere has a two-dimensional unstable manifold and a
one-dimensional stable manifold. The orbits going out from this
point into the finite region of the phase space contain profiles
satisfying $f(0)=0$ and $f(\xi)\sim K\xi^{1/m}$ as $\xi\to0$ in a
right-neighborhood of $\xi=0$.
\end{lemma}
\begin{proof}
Using once more the results in \cite[Section 3.10]{Pe}, we deduce
that the flow in a neighborhood of the point $Q_5$ is topologically
equivalent to the flow of the same system \eqref{systinf1} but this
time in a neighborhood of the critical point with coordinates
$(y,z,w)=(1/m,0,0)$ in the notation of \eqref{systinf1}. Moreover,
since $X\sim mY$ when approaching $Q_5$,
$$
\dot{X}=mX[(m-1)Y-2X]\sim-m^2(m+1)Y^2<0,
$$
hence we have to choose again the minus sign in the system
\eqref{systinf1}. The linearization of \eqref{systinf1} near $Q_5$
has thus the matrix
$$
M(Q_5)=\left(
         \begin{array}{ccc}
           -m & -1 & -\beta/m+\alpha \\
           0 & m(\sigma+1)+p & 0 \\
           0 & 0 & m+1 \\
         \end{array}
       \right),
$$
whence we obtain a two-dimensional unstable manifold and a
one-dimensional stable manifold. An easy analysis of the
eigenvectors of the matrix $M(Q_5)$ shows that the orbits going out
from $Q_5$ on the unstable manifold go to the finite part of the
phase-space. Passing now to profiles, since $X\sim mY$ in
a neighborhood of $Q_5$, whence by \eqref{PSchange2},
$$
f^{m-1}(\xi)\xi^{-2}\sim mf^{m-2}(\xi)f'(\xi)\xi^{-1},
$$
and by direct integration we obtain
$$
f(\xi)\sim K\xi^{1/m} \ \ {\rm as} \ \xi\to0, \qquad {\rm for} \
K>0,
$$
as stated.
\end{proof}
We still have to analyze the orbits connecting to or from the critical point $Q_4$. This point is a non-hyperbolic one and the local analysis near it might be a difficult task if performed in standard ways. However, we can conclude that \emph{there are no orbits} entering $Q_4$ from the finite part of the phase space or going out of $Q_4$ into the finite part of the phase space. This is a consequence of the following result.
\begin{lemma}\label{lem.Q4}
There are no solutions to \eqref{SSODE} such that
\begin{equation}\label{interm34}
\lim\limits_{\xi\to\infty}\xi^{\sigma}f(\xi)^{p-1}=+\infty.
\end{equation}
\end{lemma}
\begin{proof}
Assume for contradiction that \eqref{SSODE} admits such a solution $f$. As an immediate remark, it follows that $f(\xi)>0$ for any $\xi\in(R,\infty)$ with some $R>0$ sufficiently large. The rest of the proof is divided into several steps.

\medskip

\noindent \textbf{Step 1.} In a first step we show that $f$ is monotonic in a neighborhood of infinity. Indeed, assume that $(\xi_{0,n})_{n\geq1}$ is a sequence of local minima for $f$ such that $\xi_{0,n}\to\infty$ as $n\to\infty$. Evaluating \eqref{SSODE} at $\xi=\xi_{0,n}$ and taking into account that $f'(\xi_{0,n})=0$, $(f^m)''(\xi_{0,n})\geq0$, we infer that
$$
\xi_{0,n}^{\sigma}f(\xi_{0,n})^p\leq\alpha f(\xi_{0,n}), \quad n\geq1,
$$
whence $\xi_{0,n}^{\sigma}f(\xi_{0,n})^{p-1}\leq\alpha$, which contradicts \eqref{interm34} as we assumed that the sequence $(\xi_{0,n})_{n\geq1}$ was unbounded. It thus follows that there exists $R>0$ such that $f$ has no local minima in $(R,\infty)$, that is, it is either increasing or having at some point a local maxima and then becoming decreasing in a neighborhood of infinity.

\medskip

\noindent \textbf{Step 2.} Step 1 implies that there exists $\lim\limits_{\xi\to\infty}f(\xi)=:L\geq0$. If $L\in(0,\infty)$, it follows by applying twice a standard calculus lemma (stated explicitly in \cite[Lemma 2.9]{IL13a}) either to the function $f(\xi)-L$ if $f$ decreasing near infinity or to the function $L-f(\xi)$ if $f$ increasing near infinity, that there exists a subsequence $\{\xi_n\}_{n\geq1}$ such that
$$
\lim\limits_{n\to\infty}(f^m)''(\xi_n)=\lim\limits_{n\to\infty}\xi_nf'(\xi_n)=0, \quad \lim\limits_{n\to\infty}\xi_n=\infty.
$$
Evaluating then \eqref{SSODE} at $\xi=\xi_n$ for any $n\geq1$, we get that
$$
\lim\limits_{n\to\infty}(\xi_n^{\sigma}f(\xi_n)^{p}-\alpha f(\xi_n))=0,
$$
which is once more a contradiction with the assumption \eqref{interm34}.

\medskip

\noindent \textbf{Step 3.} If $L=+\infty$, then we infer from Step 1 that $f$ is increasing on an interval $(R,\infty)$, and we further deduce from \eqref{SSODE} and \eqref{interm34} that $(f^m)''(\xi)\to-\infty$ as $\xi\to\infty$. By straightforward calculus techniques we reach a contradiction.

\medskip

\noindent \textbf{Step 4.} If $L=0$, then $f(\xi)$ decreases to zero in an interval of the form $(R,\infty)$ according to Step 1. Since
$$
(f^m)''(\xi)=mf^{m-1}(\xi)f''(\xi)+m(m-1)f^{m-2}(\xi)f'(\xi)^2\geq mf^{m-1}(\xi)f''(\xi),
$$
there exists a sequence of intervals $[\xi_n^1,\xi_n^2]$ with $\xi_n^1\to\infty$ as $n\to\infty$, such that $f''(\xi)\geq0$ in $[\xi_n^1,\xi_n^2]$ (and the same for $(f^m)''(\xi)$), which together with \eqref{SSODE} gives
\begin{equation}\label{interm35}
\left[\xi^{\sigma}f(\xi)^{p-1}-\alpha\right]f(\xi)+\beta\xi f'(\xi)\leq0, \quad \xi\in[\xi_n^1,\xi_n^2].
\end{equation}
We further derive from \eqref{interm34} and \eqref{interm35} that for any $K>0$ large there exists an integer $n=n(K)$ such that
$$
\frac{f'(\xi)}{f(\xi)}\leq-\frac{K}{\beta\xi}, \quad \xi\in[\xi_n^1,\xi_n^2], \ n\geq n(K),
$$
but it can be easily proved that the latter contradicts the fact that $f(\xi)$ decays to zero slower than a fixed tail.
\end{proof}
Going back to the phase space in variables $(X,Y,Z)$, a profile contained in an orbit entering $Q_4$ means that $Z=\xi^{\sigma}f^{p-1}(\xi)\to\infty$ along this orbit, and such an orbit cannot contain any profile according to Lemma \ref{lem.Q4}. We are now prepared to perform the global analysis of the phase space and prove our main results.

\section{Existence of compactly supported self-similar
solutions}\label{sec.exist}

This section is devoted to the proof of Theorem \ref{th.exist}. We
will thus consider good profiles with interface according to
Definition \ref{def1}, either satisfying property (P1) or property
(P2). Let us recall that for the homogeneous case $\sigma=0$,
existence (and uniqueness in the one-dimensional case) of such
profiles is well-known, see for example \cite[Theorem 2, p.
187]{S4}. In this case, in dimension $N=1$ there exists a unique
compactly supported profile $f$ such that $f(0)>0$, $f'(0)=0$ and
$f$ strictly decreasing in $(0,\xi_0)$, where $\xi_0\in(0,\infty)$
is the first point such that $f(\xi_0)=0$. However, the presence of
the weighted reaction term in Eq. \eqref{eq1} lead to some striking
differences with respect to the homogeneous case $\sigma=0$. The
first one of them is illustrated by the following
\begin{lemma}\label{lem.monot}
Let $f$ be a solution to \eqref{SSODE} with $\sigma>0$ such that
$f(0)=a>0$ and $f'(0)=0$. Then $f'(\xi)>0$ in a right neighborhood
of $\xi=0$. In particular, there are no decreasing self-similar good
profiles for Eq. \eqref{eq1}.
\end{lemma}
\begin{proof}
We first notice that
\begin{equation}\label{interm8}
(f^m)'(0)=mf^{m-1}(0)f'(0)=0,
\end{equation}
and that $\lim\limits_{\xi\to0}|\xi|^{\sigma}f^p(\xi)=0$, for
$\sigma>0$. Thus, taking limits as $\xi\to0$ in \eqref{SSODE} we
find that
\begin{equation}\label{interm9}
\lim\limits_{\xi\to0}(f^m)''(\xi)=\alpha f(0)=\alpha a>0.
\end{equation}
Combining \eqref{interm8} and \eqref{interm9} leads to
$$
0<(f^m)'(\xi)=mf^{m-1}(\xi)f'(\xi)
$$
for any $\xi$ in a sufficiently small right neighborhood of the
origin, thus also $f'(\xi)>0$ in this neighborhood, as stated.
\end{proof}
Related to the previous lemma, one can prove an easy, but still very
important property related to the maximums of the profiles. More
precisely:
\begin{lemma}\label{lem.max}
Let $f$ be a solution to \eqref{SSODE} and $\xi_0\in(0,\infty)$ be a
local maximum point for $f$. Then we have:
\begin{equation}\label{eq.max}
f(\xi_0)\geq\alpha^{1/(p-1)}\xi_0^{-\sigma/(p-1)}.
\end{equation}
\end{lemma}
\begin{proof}
This follows easily by evaluating Eq. \eqref{SSODE} at the point
$\xi=\xi_0$, taking into account that $f'(\xi_0)=0$ and
$(f^m)''(\xi_0)\leq0$. This implies that
$$
\alpha f(\xi_0)\leq\xi_0^{\sigma}f(\xi_0)^{p},
$$
which, together with the fact that $f(\xi_0)>0$, readily leads to
\eqref{eq.max}.
\end{proof}
In order to establish the existence of self-similar profiles for
\eqref{eq1} satisfying properties (P1) and (P2), a first important
step is to establish a uniqueness result for any given interface
point. We thus prove:
\begin{proposition}\label{prop.uniqFB}
For any given $\xi_0\in(0,\infty)$, there exists a unique solution
$f$ to \eqref{SSODE} such that
$$
f(\xi)>0 \ {\rm for} \ \xi\in(0,\xi_0), \qquad f(\xi)=0 \ {\rm for}
\ \xi\geq\xi_0, \qquad \lim\limits_{\xi\to\xi_0}(f^m)'(\xi)=0,
$$
or in other words, a unique solution to \eqref{SSODE} having $\xi_0$
as interface point.
\end{proposition}
\begin{proof}
Let $\xi_0\in(0,\infty)$ be given. Throughout this proof, we will be
using the phase space associated to the system \eqref{PSsyst1} and
recall that the interface points and profiles were seen in this
phase space as critical points on the half-line of equations $x=0$, $my+\beta
z=0$ with $z>0$ and $y<0$ and orbits entering these points.
Recalling \eqref{PSchange1}, the critical point corresponding to the
given $\xi_0>0$ has the coordinates $(0,-\beta\xi_0/m,\xi_0)$. The
linearization of the system \eqref{PSsyst1} near this critical point
has the matrix
$$
M=\left(
    \begin{array}{ccc}
      -(m-1)\beta\xi_0 & 0 & 0 \\
      \alpha & \beta\xi_0 & \frac{\beta^2\xi_0}{m} \\
      m & 0 & 0 \\
    \end{array}
  \right)
$$
which shows that the critical point has a one-dimensional stable
manifold, a one-dimensional unstable manifold and a one-dimensional
center manifold. Similarly as in the analysis performed in \cite[Lemma 2.2]{IS1}, one can readily prove that
the center and unstable manifolds are unique and lie in the invariant plane $x=0$,
while there is a unique connection entering this critical point from
outside the plane $x=0$ which contains the profile we look for.
\end{proof}
\noindent \textbf{Remark.} One can also establish how a compactly
supported profile $f$ approaches its interface point $\xi_0$.
Indeed, we obtain that
$$
(f^{m-2}f')(\xi)\sim-\frac{\beta}{m}\xi, \qquad {\rm as} \
\xi\to\xi_0,
$$
whence by integration,
$$
f(\xi)\sim\left[C-\frac{\beta(m-1)}{2m}\xi^2\right]^{1/(m-1)}, \ \
C=\frac{\beta(m-1)}{2m}\xi_0^2, \qquad {\rm as} \ \xi\to\xi_0.
$$
This obviously reminds of the Barenblatt solutions to the standard
porous medium equation (see for example \cite{ISV}).

According to Proposition \ref{prop.uniqFB}, for $\eta\in(0,\infty)$
given, we denote by $f_{\eta}$ the unique solution to \eqref{SSODE}
having $\eta>0$ as interface point. The next step in the existence
proof is to display a shooting method from the interface point, that
is, to trace the trajectory of $f_{\eta}$ when moving
$\eta\in(0,\infty)$, as already explained in the Introduction. The
idea of the shooting is to show that, when either $\eta>0$ is too
close to $+\infty$, or too close to 0, then $f_{\eta}$ is not a good
profile. We then have
\begin{proposition}\label{prop.far}
There exists $\eta^*\in(0,\infty)$ such that, for any $\eta>\eta^*$,
the profile $f_{\eta}$ solution to \eqref{SSODE} with interface
point at $\eta$ changes sign "backward" at some positive point. More
precisely, for any $\eta>\eta^*$, there exists $\theta\in(0,\eta)$
such that
$$
f_{\eta}(\theta)=0, \quad (f_{\eta}^m)'(\theta)>0, \quad
f_{\eta}(\xi)>0 \ {\rm for} \ \xi\in(\theta,\eta).
$$
\end{proposition}
In the proof of Proposition \ref{prop.far}, due to technical
problems, we will switch to another quadratic phase plane system,
slightly different from \eqref{PSsyst2}. To this end, we replace $Z$
by the new variable
$$
\overline{Z}:=XZ
$$
obtaining the following autonomous system
\begin{equation}\label{PSsyst3}
\left\{\begin{array}{ll}\dot{X}=mX[(m-1)Y-2X],\\
\dot{Y}=-mY^2-\beta Y+\alpha X-mXY-\overline{Z},\\
\dot{\overline{Z}}=m\overline{Z}[(m+p-2)Y+(\sigma-2)X].\end{array}\right.
\end{equation}
Let us notice that in terms of profiles, we have
$\overline{Z}=\xi^{\sigma-2}f(\xi)^{m+p-2}$. This shows why we do not
use always in our analysis the system \eqref{PSsyst3}: in many
situations, using it will lead us to a discussion of whether
$\sigma>2$ or $\sigma<2$. But it is useful in the proof of
Proposition \ref{prop.far}.
\begin{proof}
This proof is divided into several steps. The first of them is
establishing a \emph{limit connection} which lies in the invariant
plane $X=0$ and connects with a point at space infinity. Then, in a
second step we show that the critical point at space infinity is a
node, and finally, this allows us to establish that there exist
connections in the phase space associated to \eqref{PSsyst3} but
outside the plane $X=0$, connecting the same two critical points,
which can be translated into profiles $f(\xi)$.

\medskip

\noindent\textbf{Step 1. Analysis in the invariant plane $X=0$.}
Letting $X=0$, we arrive to the following reduced phase plane:
\begin{equation}\label{interm10}
\left\{\begin{array}{ll}\dot{Y}=-mY^2-\beta Y-\overline{Z},\\
\dot{\overline{Z}}=m(m+p-2)Y\overline{Z}.\end{array}\right.
\end{equation}
In this plane, the interface point $P_1$ reduces to the critical
point with coordinates $P_1=(-\beta/m,0)$, which is a saddle point
(we omit the easy verification). We are interested in the (unique)
orbit of the plane entering $P_1$. We notice that there exists a
special curve in the phase plane associated to the system
\eqref{interm10}, which is
\begin{equation}\label{interm11}
\overline{Z}=-mY^2-\beta Y,
\end{equation}
which together with the two axis divide the plane into three regions
of monotonicity of $\overline{Z}$ as a function of $Y$: the region
"interior" to the curve \eqref{interm11}, where $\overline{Z}$
decreases with respect to $Y$, the first quadrant $\overline{Z}>0$,
$Y>0$ where again $\overline{Z}$ is decreasing with respect to $Y$,
and the region where $Y<0$ and $\overline{Z}>-mY^2-\beta Y$, in
which $\overline{Z}$ is increasing with respect to $Y$ along the
orbits. With this splitting of the plane in mind, we readily observe
that the orbit entering the saddle point $P_1$ in the plane has to
enter through the region $Y<0$, $\overline{Z}>-mY^2-\beta Y$.

We next show that the orbit entering $P_1$ in the plane has to cross
the axis $Y=0$ at a positive height $\overline{Z}$. Assume by
contradiction that this is not the case. It follows that the orbit
entering $P_1$ comes from a critical point (has a vertical
asymptote) with $Y\in(-\beta/m,0)$ and $\overline{Z}=\infty$. We can
then write that, along this curve, we have:
\begin{equation*}
\begin{split}
\frac{d\overline{Z}}{dY}&=\frac{m(m+p-2)\overline{Z}Y}{-\overline{Z}-mY^2-\beta
Y}=\frac{m(m+p-2)\overline{Z}(-Y)}{\overline{Z}+(mY^2+\beta Y)}\\
&\leq-\frac{m(m+p-2)\overline{Z}Y}{\overline{Z}}=-m(m+p-2)Y,
\end{split}
\end{equation*}
whence by integration we get that
$$
\overline{Z}\leq K-\frac{m(m+p-2)}{2}Y^2, \qquad {\rm for \ some } \
K>0,
$$
a contradiction with the fact that $Y\in(-\beta/m,0)$ and
$\overline{Z}\to\infty$.

Thus, the orbit entering $P_1$ will cross the axis $Y=0$ at some
finite positive height $\overline{Z}>0$, thus entering the first
quadrant where it remains decreasing forever with
$\lim\limits_{Y\to\infty}\overline{Z}(Y)=L\geq0$. In this case, for
a sufficiently large constant $K>0$, we find that
$$
\frac{d\overline{Z}}{dY}=\frac{m(m+p-2)\overline{Z}Y}{-\overline{Z}-mY^2-\beta
Y}\leq-\frac{m(m+p-2)\overline{Z}Y}{K+mY^2+\beta Y},
$$
so that $Z\leq\Theta$, where $\Theta$ is the solution to the
equation
\begin{equation}\label{interm12}
\frac{d\Theta}{dY}=-\frac{m(m+p-2)\Theta Y}{mY^2+\beta Y+K}.
\end{equation}
But an easy integration of \eqref{interm12} gives
$$
\Theta(Y)=C\exp\left[-m(m+p-2)\int_0^Y\frac{s}{ms^2+\beta
s+K}ds\right]\longrightarrow0 \quad {\rm as} \ Y\to\infty,
$$
hence also $\overline{Z}(Y)\to0$ as $Y\to\infty$.

\medskip

\noindent \textbf{Step 2. Critical point at infinity}. From Step 1
above, we deduce that on the orbit included in the plane $X=0$ one
gets $Y\to\infty$ and $Z\to0$. This suggests that this orbit comes
from a critical point at infinity whose coordinates (in generalized
sense) would be $(0,+\infty,0)$. But we already know that such point
exists for the whole phase space and is seen on the Poincar\'e
hypersphere as $Q_2=(0,1,0,0)$ according to the analysis done in
Lemma \ref{lem.inf2}. In particular, it also follows from Lemma
\ref{lem.inf2} that this critical point is an unstable node. We can
thus combine this fact with the theorem of continuity with respect
to data and parameters to prove that there exists $\epsilon_0>0$
such that for any $\epsilon\in(0,\epsilon_0)$, there exists a
connection between the critical points $Q_2$ and $P_1$ in the whole
phase space with $X(\xi)<\epsilon$ for $\xi\in(0,\infty)$. We will
analyze the profiles contained in these orbits.

\medskip

\noindent \textbf{Step 3. Analysis of the profiles.} Fix for the
moment $\epsilon\in(0,\epsilon_0)$ as explained in Step 2 above and
consider an orbit on which $X(\xi)<\epsilon$ for any
$\xi\in(0,\infty)$. Let us notice first that any profile contained
in such orbit has a maximum point. Indeed, as the orbit connects
$Q_2$ and $P_1$, it follows from Lemmas \ref{lem.2} and
\ref{lem.inf2} that there exist $0<\xi_1<\eta<\infty$ such that
$f(\xi_1)=f(\eta)=0$, $f'(\xi_1)>0$ and $f'(\eta)=0$. Thus, there
exists a maximum point $\xi_M\in(\xi_1,\eta)$ for $f$. On the one
hand, we derive from Lemma \ref{lem.max} that
$$
f(\xi_M)\geq\alpha^{1/(p-1)}\xi_M^{-\sigma/(p-1)}.
$$
On the other hand, since $X(\xi_M)<\epsilon$, we get
$$
f(\xi_M)<\epsilon\xi_M^{2/(m-1)},
$$
whence
\begin{equation}\label{interm14}
\xi_M^{\frac{2(p-1)+\sigma(m-1)}{(m-1)(p-1)}}>\frac{1}{\epsilon}\alpha^{1/(p-1)}.
\end{equation}
The estimate \eqref{interm14} shows that, by making
$\epsilon\in(0,\epsilon_0)$ as small as we want, the point $\xi_M$
(and then, also the interface point of $f$, $\eta>\xi_M$) can be as
large as we want.

We next easily remark that the following transformation
$$
X=xz^{-2}, \qquad Y=yz^{-1}, \qquad
\overline{Z}=x^{1+\frac{p-1}{m-1}}z^{\sigma-2},
$$
is a diffeomorphism between the following regions of the phase
planes associated to the systems \eqref{PSsyst3}, respectively
\eqref{PSsyst1}:
$$
\left\{X>0, Y\in\real, \overline{Z}>0\right\}\longmapsto\left\{x>0,
y\in\real, z>0\right\},
$$
mapping the orbits entering $P_1$ in the phase plane associated to
the system \eqref{PSsyst3} into the unique orbits entering the
points on the critical line $x=0$, $my+\beta z=0$ in the phase-space
associated to the system \eqref{PSsyst1} (we leave to the reader the
easy verification of this fact). Thus, the orbits entering $P_1$
with $X(\xi)<\epsilon$ analyzed previously are mapped into orbits
entering a point of the critical line $x=0$, $my+\beta z=0$ of the
system \eqref{PSsyst1} with $z$ as large as we want. Since the set of interface points $\eta\in(0,\infty)$ such
that $f_{\eta}$ belongs to an orbit connecting $Q_2$ to $P_1$ is an open set, we get that all the profiles with interface in some neighborhood of infinity belong to connections between $Q_2$ and $P_1$ in the phase space associated to the system \eqref{PSsyst3}. We reach the conclusion by Lemmas \ref{lem.2} and \ref{lem.inf2}.
\end{proof}
We are now ready to study the behavior of the profiles whose
interface lies very close to the origin.
\begin{proposition}\label{prop.close}
There exists $\eta_*\in(0,\infty)$ such that, for any
$0<\eta<\eta_*$, the profile $f_{\eta}$ solution to \eqref{SSODE}
with interface point at $\eta$ satisfies
$$
f_{\eta}(0)=a>0, \qquad f_{\eta}'(0)<0, \qquad f_{\eta}(\xi)>0  \
{\rm and \ decreasing \ for \ } \xi\in(0,\eta).
$$
\end{proposition}
\begin{proof}
This proof is divided into several steps following a similar
strategy as in the proof of Proposition \ref{prop.far}. The first of
them is establishing a \emph{limit connection} which lies in the
invariant plane $Z=0$ and connects with a point at space infinity.
In a second step we show that the critical point at space infinity
is an unstable node, which allows us to establish that there exist
orbits in the phase space associated to \eqref{PSsyst2} outside the
plane $Z=0$, connecting the same two critical points. Finally, we
undo the change of variable to characterize the profiles contained
in such orbits.

\medskip

\noindent\textbf{Step 1. Analysis in the invariant plane $Z=0$.}
Letting $Z=0$ in \eqref{PSsyst2}, we arrive to the following reduced
phase plane:
\begin{equation}\label{interm15}
\left\{\begin{array}{ll}\dot{X}=mX[(m-1)Y-2X],\\
\dot{Y}=-mY^2-\beta Y-mXY+\alpha X.\end{array}\right.
\end{equation}
In this plane, the interface point $P_1$ reduces to the critical
point with coordinates $P_1=(0,-\beta/m)$, which is a saddle point
(we omit the easy verification). We are interested in the (unique)
orbit of the plane entering $P_1$. We notice that there exists a
special curve in the phase plane associated to the system
\eqref{interm15}, given by the equation $dY/dX=0$, that is
\begin{equation}\label{interm16}
X=\frac{mY^2+\beta Y}{\alpha-mY},
\end{equation}
which passes through all the three finite critical points of the
system \eqref{interm15}, corresponding to $O$, $P_1$ and $P_2$ in
the big phase space, and having a horizontal asymptote at
$Y=\alpha/m$. The curve \eqref{interm16} together with the two axis
and the line $(m-1)Y-2X=0$ divide the plane into several regions of
monotonicity of $Y$ as a function of $X$ according to the sign of
the quantity
$$
\frac{dY}{dX}=-\frac{mY^2+\beta Y+mXY-\alpha X}{mX[(m-1)Y-2X]}.
$$
In particular, the fourth quadrant $\{X>0,\,Y<0\}$ is divided into
two regions of monotonicity: one "inside the curve"
\eqref{interm16}, where $dY/dX<0$, and the other "below the curve"
\eqref{interm16}, where $dY/dX>0$. The orbit entering $P_1$ in the
plane $Z=0$ has to do it through the region "inside the curve"
\eqref{interm16} due to monotonicity reasons. Indeed, assume by
contradiction that the orbit enters $P_1$ from the region below the
curve where $dY/dX>0$. Then on the orbit, $Y$ increases with respect
to $X$, which is a contradiction as the orbit would come from the
region $\{X>0,\,Y<-\beta/m\}$ to approach the point $X=0$,
$Y=-\beta/m$. Thus, the unique orbit entering $P_1$ comes from the
interior of the curve \eqref{interm16}, where $dY/dX<0$, hence along
this orbit, variable $Y$ decreases as $X$ increases. It is also
obvious that the orbit entering $P_1$ cannot cross the curve
\eqref{interm16} at some later point, thus $Y$ decreases with $X$
forever. In particular, there exists a limit $L>\beta/m$ such that
$$
Y=Y(X)\to-L, \qquad {\rm as} \ X\to\infty.
$$
Assume for contradiction that $L<\infty$. Then, for $X$ sufficiently
large (and $Y(X)$ then sufficiently close to $-L$), we have
$$
\left(\alpha+\frac{m}{2}L\right)X\leq(\alpha-mY)X-(mY^2+\beta
Y)\leq(\alpha+mL)X,
$$
and since $Y<0$ all along the orbit we are dealing with, we infer
that
$$
\frac{dY}{dX}\leq-\frac{(\alpha+mL/2)X}{2mX^2}=-\frac{2\alpha+mL}{4m}\frac{1}{X}.
$$
Integrating the previous differential inequality between $X=1$ and
$X=X_0>1$, we get
\begin{equation}\label{interm17}
Y(X_0)-Y(1)\leq-\frac{2\alpha+mL}{4m}\log X_0.
\end{equation}
As $X_0>1$ is arbitrarily chosen, passing to the limit as
$X_0\to\infty$ in \eqref{interm17} we obtain
$$
-L\leq Y(1)-\frac{2\alpha+mL}{4m}\lim\limits_{X_0\to\infty}\log
X_0=-\infty,
$$
and a contradiction. It follows that $Y\to-\infty$ on the orbit
entering $P_1$ contained in the plane $Z=0$.

\medskip

\noindent \textbf{Step 2. Critical point at infinity.} By the
previous analysis, it follows that the orbit entering $P_1$
contained in the plane $Z=0$ of \eqref{PSsyst2} comes from a
critical point at infinity for which $Z=0$. From the classification
performed in Lemmas \ref{lem.inf1}, \ref{lem.inf2} and
\ref{lem.inf3}, we find that the only critical point at infinity
from which it may come is $Q_1$, whose analysis is performed in
Lemma \ref{lem.inf1}. The fact that $Q_1$ is an unstable node
together with the theorem of continuity with respect to data and
parameters give that there exists $\epsilon_0>0$ such that for any
$\epsilon\in(0,\epsilon_0)$, there exists a connection between the
critical points $Q_1$ and $P_1$ in the whole phase space with
$Z(\xi)<\epsilon$ for $\xi\in(0,\infty)$. We will analyze next the
profiles contained in these orbits.

\medskip

\noindent \textbf{Step 3. Analysis of the profiles.} Since along the
orbits under consideration, we have $Z(\xi)<\epsilon$, it first
follows that
$$
\xi^{\sigma}f^{p-1}(\xi)<\epsilon, \qquad {\rm for \ any} \
\xi\in(0,\infty).
$$
In particular, choosing $\epsilon\in(0,\alpha)$ sufficiently small,
it follows from Lemma \ref{lem.max} that $f(\xi)$ decreases for
$\xi\in(0,\eta)$, where we recall that we denote by $\eta$ the
interface point of the profile $f$. Indeed, for any
$\epsilon\in(0,\alpha)$, the inequality \eqref{eq.max} is false,
hence $f$ does not admit local maximum points. Moreover, coming back
to the original equation \eqref{SSODE}, one finds that
\begin{equation}\label{interm17}
(f^m)''(\xi)=f(\xi)(\alpha-\xi^{\sigma}f^{p-1}(\xi))-\beta\xi
f'(\xi)\geq Lf(\xi), \quad L=\alpha-\epsilon,
\end{equation}
where we took into account that $f'(\xi)\leq0$ for $\xi\in(0,\eta)$.
By letting first $g=f^m$ \eqref{interm17} and then doing a standard
integration in two steps in the resulting differential inequality
(taking into account that $g'<0$), we find that
\begin{equation*}
f(\xi)\geq K(\eta-\xi)^{2/(m-1)}, \qquad {\rm for} \ \xi\in(0,\eta),
\end{equation*}
for a constant $K>0$ (that can be explicitly calculated, but we omit
its expression) depending on $m$, $\alpha-\epsilon$ and the initial
value $a=f(0)$. In particular,
$$
\xi^{\sigma}K^{p-1}(\eta-\xi)^{2(p-1)/(m-1)}<\epsilon, \qquad {\rm
for} \ \xi\in(0,\eta),
$$
and we infer by evaluating this inequality at $\xi=\eta/2$ that
\begin{equation}\label{interm18}
K^{p-1}\left(\frac{\eta}{2}\right)^{\sigma+2(p-1)/(m-1)}<\epsilon.
\end{equation}
We deduce from \eqref{interm18} that $\eta$ is forced to be as small
as we want, provided $\epsilon>0$ is chosen sufficiently small.
Noticing next that the following transformation:
$$
X=xz^{-2}, \qquad Y=yz^{-1}, \qquad Z=x^{(p-1)/(m-1)}z^{\sigma},
$$
is a diffeomorphism between the following regions of the phase
planes associated to the systems \eqref{PSsyst2}, respectively
\eqref{PSsyst1}:
$$
\left\{X>0, Y\in\real, Z>0\right\}\longmapsto\left\{x>0, y\in\real,
z>0\right\},
$$
mapping the orbits entering $P_1$ in the phase plane associated to
the system \eqref{PSsyst2} into the unique orbits entering the
points on the critical line $x=0$, $my+\beta z=0$ in the phase-space
associated to the system \eqref{PSsyst1}, we end up the proof as in
the last part of Step 3 in the proof of Proposition \ref{prop.far},
we leave these details to the reader.
\end{proof}
With all these preparations, we are now ready to prove the existence
theorem.
\begin{proof}[Proof of Theorem \ref{th.exist}]
Denote by $A_{-}$ the set of $\eta\in(0,\infty)$ such that the
profile $f_{\eta}$ with interface exactly at $\xi=\eta$ satisfies
$$
f_{\eta}(0)=a>0, \qquad f_{\eta}'(0)<0,
$$
and by $A_+$ the set of $\eta\in(0,\infty)$ such that there exists
$\xi_0\in(0,\eta)$ with
$$
f_{\eta}(\xi_0)=0, \qquad f_{\eta}'(\xi_0)>0.
$$
It is easy to see (by continuity with respect to the parameters)
that $A_-$ and $A_+$ are both open sets. Moreover, Proposition
\ref{prop.close} insures that $A_-\neq\emptyset$, while Proposition
\ref{prop.far} insures that $A_+\neq\emptyset$ and there exists an
interval $(\eta^*,\infty)\subseteq A_+$. Let then
$$
\eta_0=\sup A_-<\eta^*<\infty.
$$
We want to analyze the behavior of the unique profile
$f_{\eta_0}(\xi)$ having an interface at $\xi=\eta_0$. First of all,
since both $A_+$ and $A_-$ are open sets, $\eta_0$ does not belong
to any of them.

Moreover, $f_{\eta_0}$ cannot have a vertical asymptote at $\xi=0$, since according to the local analysis in Section \ref{sec.local}, there is no such behavior in the phase space.

Moreover, there is no point $\xi_1\in(0,\eta_0)$ such that
$f_{\eta_0}(\xi_1)=0$ and $f_{\eta_0}>0$ in $(\xi_1,\eta_0)$.
Indeed, assuming for contradiction that such $\xi_1$ exists,

\indent $\bullet$ either $f_{\eta_0}'(\xi_1)>0$, meaning that
$\eta_0\in A_+$. As $A_+$ is open, there exists $\delta>0$ such that
$(\eta_0-\delta,\eta_0)\subseteq A_+$. But by the definition of
$\eta_0$ as supremum of $A_-$, it follows that $A_+\cap
A_-\neq\emptyset$, and a contradiction.

\indent $\bullet$ or $f_{\eta_0}'(\xi_1)=0$, which is a
contradiction as such behavior does not exist in the phase space
system \eqref{PSsyst2} as it follows from Section \ref{sec.local}.

From these facts and the continuity with respect to the parameter we
deduce that $f_{\eta_0}$ intersects the axis $\xi=0$ either at the
origin (which is one case of "good solution") or at some point
$a>0$. In the latter case, we easily find by continuity that
$f_{\eta_0}'(0)\leq0$, and since $\eta_0\not\in A_-$, the case
$f_{\eta_0}'(0)<0$ is impossible, whence we reach the conclusion.
\end{proof}

\section{Self-similar blow up profiles for $\sigma$
small}\label{sec.small}

In this section, we deal with the blow up profiles to Eq.
\eqref{eq1} for $\sigma>0$ sufficiently small, in particular proving
Theorems \ref{th.small} and \ref{th.decay}. Let us recall that, by
Theorem \ref{th.exist}, there exists at least a good profile with
interface, that is, a solution $f$ to \eqref{SSODE} satisfying one
of the hypothesis (P1) and (P2). Our next goal is to prove that, for
$\sigma>0$ sufficiently small, only condition (P1) occurs.
\begin{proof}[Proof of Theorem \ref{th.small}]
Assume for contradiction that the statement of Theorem
\ref{th.small} is not true, that is, there exists a sequence
$\{\sigma_n\}_{n}$ such that $\sigma_n\to0$ as $n\to\infty$ and
corresponding good profiles with interface of type (P2) denoted by
$f_n$ with corresponding interface points $\eta_n\in(0,\infty)$. We
first prove the following important technical step.
\begin{lemma}\label{lem.unibound}
There exists $\sigma_0>0$ and constants $C_1$, $C_2>0$ which are
independent of $\sigma$ such that $C_1<\eta<C_2$ for any good
profile $f_{\eta}$ to \eqref{SSODE} with exponent
$\sigma\in(0,\sigma_0)$ and with interface at point
$\eta\in(0,\infty)$.
\end{lemma}
\begin{proof}
For $\sigma=0$, there exists a unique good profile $f_{\eta_0}$ with
interface at some $\eta_0\in(0,\infty)$ \cite[Theorem 2, p. 187 and
Lemma 3, pp. 264-265]{S4}. We also infer from Propositions
\ref{prop.close} and \ref{prop.far} that for any $\sigma>0$ fixed
there exist constants $C_1(\sigma)<C_2(\sigma)\in(0,\infty)$ such
that for any good profile $f_{\eta}$, his interface point $\eta$
satisfies
$$
C_1(\sigma)<\eta<C_2(\sigma).
$$
Thus we readily reach the conclusion from the continuity with
respect to the parameter $\sigma$ in the system \eqref{PSsyst1}
\cite[Theorem 3, Chapter 15]{HS} (valid up to $\sigma=0$) and the
fact that $C_1(0)=C_2(0)=\eta_0$.
\end{proof}
We obtain from Lemma \ref{lem.unibound} that there exists a positive
integer $n_0$ such that
\begin{equation}\label{interm21}
C_1<\eta_n<C_2, \qquad {\rm for} \ n\geq n_0.
\end{equation}
Without losing on generality, we may relabel the sequences
$\{\sigma_n\}_{n}$ and $\{f_n\}_{n}$ such that $n_0=1$, thus
\eqref{interm21} holds true for any $n\geq1$. Moreover, since
$\eta_n$ is uniformly bounded, it has a convergent subsequence, so
that we can once more relabel (retaining only a subsequence) the
sequences $\{\sigma_n\}_{n}$ and $\{f_n\}_{n}$ such that
$\eta_n\to\eta_0$ as $n\to\infty$. We need now a second technical
step.
\begin{lemma}\label{lem.unibound2}
There exists a positive integer $n_1\geq1$ such that the sequence
$\{f_n\}_{n}$ is uniformly bounded (independent of $\sigma$) for
$n\geq n_1$.
\end{lemma}
\begin{proof}
We use once more the continuity with respect to the parameter
$\sigma>0$ (and to the data) in the non-autonomous first order
system \eqref{PSsyst1} (see for example \cite[Theorem 3, Chapter
15]{HS}). Let $f_0$ be the unique profile to \eqref{SSODE} with
$\sigma=0$ having interface at $\xi=\eta_0$. At a formal level, we
know that $f_0\in L^{\infty}[0,\infty)$ and by the above mentioned
result, for any given $\delta>0$, there exist $n_{\delta}\geq1$ and
$K>0$ such that for any $n\geq n_{\delta}$ and
$\xi\in[0,\max\{\eta_0,\eta_n\}]$ we have
\begin{equation}\label{intermX}
|f_n(\xi)-f_0(\xi)|\leq\delta\left(K\exp{|\xi-\eta_0|}-1\right).
\end{equation}
This is only formal, as it cannot be applied rigorously when taking
as initial data a critical point in \eqref{PSsyst1}. But it can be
made rigorous by taking a very small ball $B(\eta_0,r_0)$ around the
point $\xi=\eta_0$, which contains an infinity of the points
$\eta_n$, and taking as data the intersections of the trajectories
in the phase plane associated to the system \eqref{PSsyst1} with
this ball. The conclusion follows obviously from \eqref{intermX} and
the fact that $f_0\in L^{\infty}[0,\infty)$.
\end{proof}
Relabeling the sequences such that $n_1=1$ in order to simplify the
notation, we deduce from Lemma \ref{lem.unibound2} and \eqref{SSODE}
that
$$
(f_n^m)''(\xi)+\beta_n\xi f_n'(\xi)=\alpha_n
f_n(\xi)-\xi^{\sigma_n}f_n(\xi)^p\in[-K,K], \qquad \xi\in[0,C_2]
$$
for some $K>0$ sufficiently large but independent of $\sigma$. Here
and in the sequel, we denote by
$$
\alpha_n:=\frac{\sigma_n+2}{2(p-1)+\sigma_n(m-1)}, \quad
\beta_n:=\frac{m-p}{2(p-1)+\sigma_n(m-1)}
$$
the self-similarity exponents corresponding to our sequence
$\{\sigma_n\}_n$. We thus find that
$$
-K\leq(f_n^m)''(\xi)+\beta_n\xi f_n'(\xi)\leq K, \qquad
\xi\in[0,C_2],
$$
whence by integration by parts on $[0,\xi]$,
\begin{equation}\label{interm22}
|(f^m)'(\xi)|\leq K\xi+\left|\beta_n\xi f_n(\xi)-\int_0^{\xi}\beta_n
f_n(s)\,ds\right|\leq\overline{K}, \qquad \xi\in[0,C_2],
\end{equation}
for $\overline{K}>0$ independent of $\sigma$, where we have used
once more for the last inequality in \eqref{interm22} Lemmas
\ref{lem.unibound} and \ref{lem.unibound2}. Since
$$
f_n'(\xi)=\frac{1}{m}(f_n^m)'(\xi)f_n(\xi)^{1-m}, \qquad
\xi\in[0,C_2]
$$
it follows that $(f_n)'$ is uniformly bounded (independent of
$\sigma_n$) far from $\xi=0$ and $\xi=\eta_n$. Notice for example
that $(f_n)'$ may not be uniformly bounded at $\xi=0$ or
$\xi=\eta_n$ (a closer inspection of the behavior of $f_n$ shows
that this is only true when $1<m\leq2$). From \eqref{SSODE} we also
deduce that
\begin{equation}\label{interm23}
(f_n^m)''(\xi)=\alpha_n
f_n(\xi)-\beta_n\xi(f_n)'(\xi)+\xi^{\sigma_n}f_n(\xi)^p, \qquad
\xi\in[0,C_2],
\end{equation}
whence on the one hand, recalling that we are under the assumption
that $f_n(0)=0$, $(f_n^m)'(0)=0$ and that there are only two
possible behaviors for $f$ as $\xi\to0$, established in Lemmas
\ref{lem.1} and \ref{lem.3}, we infer that $(f_n^m)''(0)=0$ whatever
the behavior of $f_n$ is (among the two possibilities). On the other
hand, we also get from \eqref{interm23} that $(f_n^m)''$ is
uniformly bounded on any interval $[0,\xi]$ with $\xi<\eta_n$. Even
more, we derive from \eqref{interm23} that $(f_n^m)''$ is uniformly
Holder in any interval $[0,\xi]$ with $\xi<\eta_n$, as the right
hand side of \eqref{interm23} has obviously this property.

Using the Arzel\`a-Ascoli Theorem (and in particular the compact
embedding of the Holder space $C^{2,\gamma}[0,\xi]$ in $C^2[0,\xi]$
for any $\gamma\in(0,1)$), we find that there exist functions $g_1$,
$g_2$ and $g_3$ such that for any $0<\xi_1<\xi_2<C_2$, the following
hold true:

\indent $\bullet$ $f_n\longmapsto g_1$ uniformly in $[0,\xi_2]$,

\indent $\bullet$ $(f_n)'\longmapsto g_2$ uniformly in
$[\xi_1,\xi_2]$

\indent $\bullet$ $(f_n^m)''\longmapsto g_3$ uniformly in
$[0,\xi_2]$.

It is now easy to identify the functions $g_1$, $g_2$ and $g_3$. In
a first step, the first convergence in the list above holds also
pointwisely for any $\xi\in(0,\eta_0)$, where we recall that
$\eta_0=\lim\limits_{n\to\infty}\eta_n$ according to a convention we
made at the beginning of the proof. Thus, one readily gets that
$g_1(0)=g_1(\eta_0)=0$ and that $g_1$ is supported in $[0,\eta_0]$.
Moreover, using the uniform convergence and the definition of the
derivative, we have for any $\xi\in(0,\eta_0)$:
\begin{equation*}
\begin{split}
g_2(\xi)&=\lim\limits_{n\to\infty}f_n'(\xi)=\lim\limits_{n\to\infty}\left(\lim\limits_{h\to0}\frac{f_n(\xi+h)-f_n(\xi)}{h}\right)\\
&=\lim\limits_{h\to0}\left(\lim\limits_{n\to\infty}\frac{f_n(\xi+h)-f_n(\xi)}{h}\right)=\lim\limits_{h\to0}\frac{g_1(\xi+h)-g_1(\xi)}{h}=g_1'(\xi).
\end{split}
\end{equation*}
By a similar argument of commuting limits, one can also show that
$$
g_3(\xi)=(g_1^m)''(\xi), \qquad \xi\in(0,\eta_0).
$$
Thus, relabeling for simplicity $g_1\equiv g$, we can pass to the
limit as $n\to\infty$ in \eqref{SSODE} (applied to $f_n$) to show
that $g$ solves the ordinary differential equation
$$
(g^m)''(\xi)+\beta(0)\xi g'(\xi)-\alpha(0)g(\xi)+g^p(\xi)=0, \qquad
\alpha(0)=\frac{1}{p-1}, \ \beta(0)=\frac{m-p}{2(p-1)},
$$
for any $\xi\in(0,\eta_0)$. Moreover, ${\rm supp}\,g=[0,\eta_0]$ and
the interface condition at $\xi=\eta_0$ is fulfilled by $g$:
$$
(g^m)'(\eta_0)=\lim\limits_{n\to\infty}(g^m)'(\eta_n)=\lim\limits_{n\to\infty}(f_n^m)'(\eta_n)=0,
$$
where the second equality was allowed by the uniform convergence of
$(f_n^m)'$ to $(g^m)'$ in any compact interval included in
$(0,\eta_0)$. It thus follows that $g$ is a solution to the
homogeneous equation corresponding to \eqref{SSODE} for $\sigma=0$,
with interface point at $\eta_0$ and with $g(0)=0$, $(g^m)'(0)=0$.

We still need to prove that the limit function $g$ is not the
trivial function. To this end, recalling that
$f_n(0)=f_n(\eta_n)=0$, let $z_n\in(0,\eta_n)$ be such that
$$
f_n(z_n)=\max\limits_{\xi\in(0,\eta_n)}f_n(\xi).
$$
We then infer from Lemma \ref{lem.max} that
$$f_n(z_n)\geq\alpha_n^{1/(p-1)}z_n^{-\sigma_n/(p-1)},$$
hence passing to the limit and taking into account the uniform
convergence on compact sets in $(0,\eta_0)$, implies
$$
\|g\|_{\infty}\geq\alpha(0)^{1/(p-1)}=\left(\frac{1}{p-1}\right)^{1/(p-1)}>0,
$$
and a contradiction to well-established results concerning the
homogeneous case $\sigma=0$, which show that such solution $g$ as
above does not exist \cite[Theorem 2, p. 187 and Lemma 3, pp.
264-265]{S4}.
\end{proof}

\section{Self-similar profiles with $f(0)=0$}\label{sec.5}

We devote this section to the study of self-similar profiles satisfying property (P2) in Definition \ref{def1}, that is, profiles $f(\xi)$ solving \eqref{SSODE} with initial conditions $f(0)=0$, $(f^m)'(0)=0$. In particular, we prove here Theorem \ref{th.large}. Let us recall here with the title of an example the following exact solution with interface to \eqref{SSODE} for the limit case $p=1$, already established in \cite[Theorem 1.2]{IS1}, satisfying property (P2) in Definition \ref{def1}
\begin{equation}\label{sol.exp}
f_1(\xi):=\xi^{2/(m-1)}\left(\frac{m-1}{2m(m+1)}-B\xi^{\sigma}\right)_{+}^{1/(m-1)}, \qquad B=\frac{(m-1)^2}{m(\sigma+2)(m\sigma+m+1)},
\end{equation}
where $\sigma=\sqrt{2(m+1)}$. We show here that such solutions exist also for $p\in(1,m)$, although they are no longer explicit. We refer the reader to our companion work \cite{IS1} for a complete study of the interesting case $p=1$.

\subsection{General properties of self-similar profiles with $f(0)=0$}

Before passing to the actual proof of Theorem \ref{th.large}, we gather here several facts related to profiles $f$ solutions to \eqref{SSODE} with $f(0)=0$, $(f^m)'(0)=0$. Besides their use in the next subsections to help with the proof of Theorem \ref{th.large}, we think that these results have interest by themselves. Throughout all this subsection, $f$ is assumed to be a self-similar profile with $f(0)=0$, $(f^m)'(0)=0$, and by $\xi_0$ we will understand any (local) maximum point of $f$.
\begin{proposition}\label{lem.unifbound}
In the previous notation and conditions

\begin{itemize}
\item[(a)] The following upper bound holds true:
\begin{equation}\label{upper}
f(\xi)\leq\left[\frac{\alpha(m-1)}{2m}\right]^{1/(m-1)}\xi^{2/(m-1)}, \qquad {\rm for \ any} \ \xi>0.
\end{equation}
\item[(b)] Recalling that $\xi_0$ is a maximum point of $f$, we have
\begin{equation}\label{unif.max}
\xi_0\geq\overline{\xi}:=\alpha^{\beta}\left[\frac{2m}{m-1}\right]^{(p-1)\beta/(m-p)}.
\end{equation}
\end{itemize}
\end{proposition}
Before proving Proposition \ref{lem.unifbound}, let us notice here that its part (a) can be seen as a uniform bound (with respect to the exponent $\sigma$) for $f$ over compact sets in $[0,\infty)$. Indeed, there is a dependence on $\alpha$ in the right hand side of \eqref{upper}, but recall that $\alpha=\alpha(\sigma)$ is uniformly bounded with respect to $\sigma$. We also infer from part (b) that a maximum point $\xi_0$ cannot lie as close as the $\xi=0$ as we want.
\begin{proof}[Proof of Proposition \ref{lem.unifbound}]
(a) We consider the plane $\{Y=\alpha/m\}$ in the phase space associated to the system \eqref{PSsyst2}. The direction of the flow of the phase space over the plane $\{Y=\alpha/m\}$ is given by the sign of the following expression:
$$
F(X,Z):=-m\left(\frac{\alpha}{m}\right)^2-\beta\frac{\alpha}{m}+\alpha X-m\frac{\alpha}{m}X-XZ=-\frac{\alpha^2+\alpha\beta}{m}-XZ<0,
$$
since $X$, $Z\geq0$. Thus, once an orbit in the phase space lies in the half-space $\{Y\leq\alpha/m\}$, it cannot cross the plane $\{Y=\alpha/m\}$. But in particular, any profile $f$ with $f(0)=0$ and $(f^m)'(0)=0$ belongs to an orbit starting from one of the points $P_0=(0,0,0)$ or $P_2=((m-1)/2m(m+1),1/m(m+1),0)$. Noticing that
$$
\sigma(m-1)+2(p-1)<(\sigma+2)(m-1)<(\sigma+2)(m+1),
$$
we readily find that both points $P_0$ and $P_2$ lie in the half-space $\{Y\leq\alpha/m\}$. This implies that along any orbit containing profiles $f$ with $f(0)=0$, $(f^m)'(0)=0$, we have $Y<\alpha/m$. Recalling the definition of $Y$ in \eqref{PSchange2}, we get
\begin{equation}\label{interm28}
(f^{m-1})'(\xi)\leq\frac{(m-1)\alpha}{m}\xi, \qquad \xi\in(0,\infty)
\end{equation}
whence by integrating over $(0,\xi)$ one readily obtains \eqref{upper}.

\medskip

%\noindent (b) We first notice that \eqref{upper} and \eqref{PSchange2} give
%$$
%X(\xi)=\xi^{-2}f^{m-1}(\xi)\leq\frac{\alpha(m-1)}{2m}, \quad \xi\geq0.
%$$
%Considering the plane $\{Y=-d\}$ with $d>\beta/m$, the direction of the flow of the phase space over the plane $\{Y=-d\}$ is given by the sign of the following expression:
%\begin{equation*}
%G(X,Z):=-md^2+\beta d+(\alpha+md)X-XZ\leq(\alpha+md)\frac{\alpha(m-1)}{2m}+\beta d-md^2<0,
%\end{equation*}
%provided $d>0$ is taken sufficiently large (independently of $\sigma$ and $f$). This shows that, once an orbit crossed $\{Y=-d\}$, it will stay forever in the half-space $\{Y<-d\}$. Since we assume that $f$ is a profile with interface, it belongs to an orbit entering $P_1=(0,-\beta/m,0)$ in the phase space, that is, this orbit can never cross the plane $\{Y=-d\}$ chosen above. It follows that $Y\geq-d$ along the orbit containing $f$, which readily implies
%$$
%(f^{m-1})'(\xi)\geq-d\xi, \quad {\rm for \ any} \ \xi\geq0.
%$$
%The above estimate, together with \eqref{interm28} and the independence of $d$ and $\alpha(m-1)/2m$ on $\sigma$ and $f$, readily give \eqref{unif.deriv}.
%
%\medskip

\noindent (b) Since $\xi_0$ is a maximum point for $f$, we gather the estimates \eqref{eq.max} and \eqref{upper} to obtain
$$
\alpha^{1/(p-1)}\xi_0^{-\sigma/(p-1)}\leq f(\xi_0)\leq\left[\frac{\alpha(m-1)}{2m}\right]^{1/(m-1)}\xi_0^{2/(m-1)},
$$
which leads to the estimate \eqref{unif.max} after straightforward calculations.
\end{proof}
The following technical result shows that for any $\sigma>0$, there exist good profiles with $f(0)=0$ and with the tail behavior \eqref{queue}.
\begin{lemma}\label{lem.tail}
For any $p\in(1,m)$ and $\sigma>0$, there exists at least an orbit in the phase space associated to \eqref{PSsyst2} connecting the points $P_0$ and $P_{\gamma_0}$. These orbits contain profiles such that $f(0)=0$, $(f^m)'(0)=0$, and behaving as in \eqref{case2} at $\xi=0$ and as in \eqref{queue} as $\xi\to\infty$.
\end{lemma}
\begin{proof}
We once more replace $Z$ by the new variable $\overline{Z}=XZ$, as we previously did in Proposition \ref{prop.far}, to obtain the autonomous system \eqref{PSsyst3}. We analyze the critical point $(X,Y,\overline{Z})=(0,0,0)$ in \eqref{PSsyst3}. Let us first notice that the above change of variable gather in the origin of the system \eqref{PSsyst3} the orbits connecting to any points $P_{\gamma}$ with $\gamma>0$. But we know from Lemma \ref{lem.4} that there is only one point among them that attracts orbits in the phase space, that is the attractor $P_{\gamma_0}$ with $\gamma_0=1/(p-1)$. Thus, it is easy to check that the new point $(X,Y,\overline{Z})=(0,0,0)$ puts together the orbits of $P_0$ and $P_{\gamma_0}$. The linearization of the system \eqref{PSsyst3} near the origin has the matrix
$$
\left(
      \begin{array}{ccc}
        0 & 0 & 0 \\
        \alpha & -\beta & -1 \\
        0 & 0 & 0 \\
      \end{array}
    \right)
$$
having thus a one-dimensional stable manifold and a two-dimensional center manifold. We analyze the center manifold following, as usual, the recipe given in \cite[Section 2.12]{Pe}. To this end, we replace $Y$ by the new variable
$$
T:=\beta Y-\alpha X+\overline{Z},
$$
deriving the system
\begin{equation*}
\left\{\begin{array}{ll}\dot{X}=\frac{m}{\beta}X[(m-1)T+X-(m-1)\overline{Z}],\\
\dot{T}=-\beta T-\frac{m}{\beta}T^2-\frac{m(2\alpha+3\beta+1)}{\beta}TX+\frac{m(m+p)}{\beta}T\overline{Z}\\
-\frac{m\alpha(\alpha+\beta+1)}{\beta}X^2+\frac{m(3\beta+2\alpha+3)}{\beta}X\overline{Z}-\frac{m(m+p-1)}{\beta}\overline{Z}^2,\\
\dot{\overline{Z}}=\frac{m}{\beta}\overline{Z}[(m+p-2)T+2X-(m+p-2)\overline{Z}].\end{array}\right.
\end{equation*}
After rather long calculations, one can find that the center manifold is given by
$$
T(X,\overline{Z})=-\frac{m\alpha(\alpha+\beta+1)}{\beta^2}X^2+\frac{m(3\beta+2\alpha+3)}{\beta^2}X\overline{Z}-\frac{m(m+p-1)}{\beta^2}\overline{Z}^2
+O(|(X,\overline{Z})|^3),
$$
and the flow on the center manifold is given (discarding the terms containing $T$, which are of higher order) by the reduced system
\begin{equation}\label{interm50}
\left\{\begin{array}{ll}\dot{X}=\frac{m}{\beta}X[X-(m-1)\overline{Z}]+O(|(X,\overline{Z})|^3),\\
\dot{\overline{Z}}=\frac{m}{\beta}\overline{Z}[2X-(m+p-2)\overline{Z}]+O(|(X,\overline{Z})|^3).\end{array}\right.
\end{equation}
In order to study the orbits near the nonhyperbolic critical point $(X,\overline{Z})$ in the system \eqref{interm50}, we first do an affine change of variable by letting:
\begin{equation}\label{interm51}
X_1:=X-(p-1)\overline{Z}, \quad Z_1:=-(p-1)\overline{Z},
\end{equation}
which transforms \eqref{interm50} into the following topologically equivalent system (in which we only keep the quadratic terms and we omit the higher order ones for simplicity)
\begin{equation}\label{interm52}
\left\{\begin{array}{ll}\dot{X_1}=\frac{m}{\beta}X_1\left[\frac{m-1}{p-1}Z_1+X_1\right],\\
\dot{Z_1}=\frac{m}{\beta}Z_1\left[2X_1+\frac{m-p}{p-1}Z_1\right],\end{array}\right.
\end{equation}
which can be written in the following form
\begin{equation*}
\frac{dZ_1}{dX_1}=f\left(\frac{Z_1}{X_1}\right), \quad f(k)=\frac{a+bk+ck^2}{d+ek+fk^2},
\end{equation*}
with coefficients
$$
a=f=0, \ b=2, \ c=\frac{m-p}{p-1}, \ d=1, \ e=\frac{m-1}{p-1}.
$$
We are thus in the framework of the general classification given in the paper \cite{Ly51}, and more precisely, noticing that
$$
f'(0)=\frac{b}{d}=2>1, \ f'(1)=\frac{c+d}{e+d}=\frac{m-1}{m+p-2}\in(0,1), \ f'(\infty)=\frac{c}{e}=\frac{m-p}{m-1}\in(0,1),
$$
the phase portrait of the system \eqref{interm52} near the origin corresponds to Case 9 in \cite[Page 177]{Ly51}, that is, a nonhyperbolic point having an elliptic sector in the quadrant $\{X_1>0, Z_1<0\}$ and a hyperbolic sector in the quadrant $\{X_1>0, Z_1>0\}$. Undoing the affine transformation \eqref{interm51} to come back to the original variables, we have
$$
X=X_1-Z_1, \quad \overline{Z}=-\frac{1}{p-1}Z_1,
$$
thus the phase portrait of the system \eqref{interm50} in a neighborhood of the origin has an elliptic sector in a part of the quadrant $\{X>0, \overline{Z}>0\}$ and a hyperbolic sector in the remaining part of this quadrant, see Figure \ref{fig5}. In particular, this implies that there are orbits going out and entering the origin of the system \eqref{PSsyst3} along the center manifold of this critical point. Since we identified the orbits entering $(X,Y,\overline{Z})=(0,0,0)$ in \eqref{PSsyst3} tangent to the center manifold of this point to the orbits entering the critical point $P_{\gamma_0}$ tangent to its center manifold in the system \eqref{PSsyst2}, it follows that there is always at least an orbit connecting $P_0$ to $P_{\gamma_0}$ in the phase space associated to the system \eqref{PSsyst2}. The conclusion follows from Lemmas \ref{lem.1} and \ref{lem.4}.
\end{proof}
We plot in Figure \ref{fig5} the phase portrait of the system \eqref{interm50} near the origin, with the two sectors, one elliptic and one hyperbolic, separated by the line $\{X=\overline{Z}\}$.

\begin{figure}[ht!]
  \begin{center}
  \includegraphics[width=9cm,height=8cm]{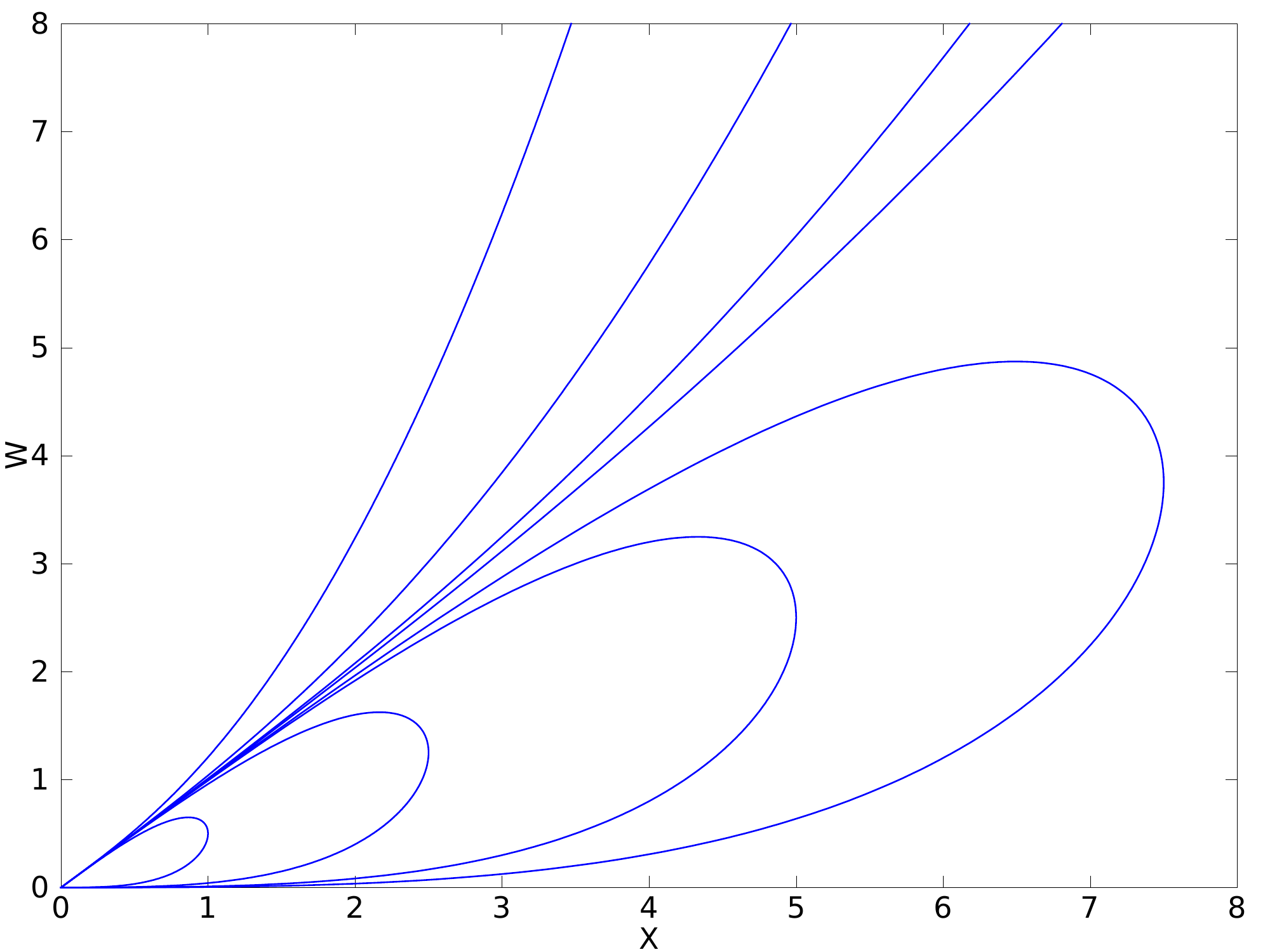}
  \end{center}
  \caption{Local behavior of the system \eqref{interm50} with one elliptic sector and one hyperbolic sector at the origin.}\label{fig5}
\end{figure}

One more interesting remark noticed first in the numerical experiments was the monotonicity of the components $X$ and $Y$ along any orbit going out of $P_2$ in the system (at least in the region where $Y>0$). This is in fact a very important feature and is established below.
\begin{lemma}\label{lem.flow2}
Consider the orbit in the phase space coming out of $P_2$ for any $\sigma>0$. Then the $X$ component is decreasing along the whole orbit and the $Y$ component is also decreasing in the region $\{Y\geq0\}$. In particular, $X<X(P_2):=(m-1)/2m(m+1)$ and $Y<Y(P_2):=1/m(m+1)$ at any point different from $P_2$.
\end{lemma}
\begin{proof}
We readily get from the analysis in Lemma \ref{lem.3} that the unique orbit coming out of $P_2$ starts with components $X$ and $Y$ decreasing in a (small) neighborhood of $P_2$. Moreover, where $Y<0$, it is obvious that $\dot{X}<0$ from the first equation in \eqref{PSsyst2}. Assume by contradiction that $X$ is not decreasing in the region $\{Y>0\}$, thus there exists a first point $\eta=\eta_1$ on the connection starting from $P_2$ where $X$ changes its monotonicity, that is, $\dot{X}(\eta_1)=0$ and $X''(\eta_1)\geq0$. That is, $(m-1)Y(\eta_1)=2X(\eta_1)$, hence
$$
0\leq X''(\eta_1)=m(m-1)X(\eta_1)Y'(\eta_1)
$$
and we infer that $Y$ should have already changed monotonicity either before $\eta_1$ or at the same point. Let then $\eta_2\in(0,\eta_1]$ be the first point where the $Y$ component of the orbit coming out of $P_2$ changes monotonicity, that is, $\dot{Y}(\eta_2)=0$ and $Y''(\eta_2)\geq0$. In a first case, if $\eta_2=\eta_1$, taking into account that also $\dot{X}(\eta_1)=0$, we compute from the second equation in \eqref{PSsyst2} that
$$
Y''(\eta_2)=-X(\eta_2)\dot{Z}(\eta_2)=-mX(\eta_2)Z(\eta_2)[(p-1)Y(\eta_2)+\sigma X(\eta_2)]<0,
$$
and a contradiction, since in the region $\{Y>0\}$ component $Z$ is strictly increasing and thus $Z(\eta_2)>0$. In a second case, if $\eta_2<\eta_1$ (the point where the monotonicity of $Y$ changes for the first time lies on the orbit before the one where $X$ does the same), we thus have $\dot{X}(\eta_2)<0$. Since $\dot{Y}(\eta_2)=0$ we get that
\begin{equation}\label{interm36}
\alpha-mY(\eta_2)-Z(\eta_2)=\frac{mY^2(\eta_2)+\beta Y(\eta_2)}{X(\eta_2)}>0,
\end{equation}
thus we deduce from \eqref{interm36} that
\begin{equation*}
\begin{split}
Y''(\eta_2)&=\alpha\dot{X}(\eta_2)-m\dot{X}(\eta_2)Y(\eta_2)-\dot{X}(\eta_2)Z(\eta_2)-X(\eta_2)\dot{Z}(\eta_2)\\
&=\dot{X}(\eta_2)[\alpha-mY(\eta_2)-Z(\eta_2)]-mX(\eta_2)Z(\eta_2)[(p-1)Y(\eta_2)+\sigma X(\eta_2)]<0,
\end{split}
\end{equation*}
and a contradiction, ending the proof.
\end{proof}
Before stating the main proposition of this subsection, we need one more technical result concerning the reduced phase-plane inside the invariant plane $\{Z=0\}$.
\begin{lemma}\label{lem.Z0}
For any $p>1$, $\sigma>0$ there exists an orbit connecting $P_0$ to $P_2$, lying in the plane $\{Z=0\}$ of the phase space associated to the system \eqref{PSsyst2}.
\end{lemma}
\begin{proof}
We restrict to the invariant plane $\{Z=0\}$, and the system \eqref{PSsyst2} reduces to
\begin{equation}\label{PPsystZ0}
\left\{\begin{array}{ll}\dot{X}=mX[(m-1)Y-2X],\\
\dot{Y}=-mY^2-\beta Y+\alpha X-mXY.\end{array}\right.
\end{equation}
We consider next two important curves in the phase plane associated to the system \eqref{PPsystZ0}. The first of them is the line of equation $(m-1)Y-2X=0$, passing through the critical points $P_0$ and $P_2$. The flow of the system \eqref{PPsystZ0} over this curve is given by the sign of the expression
$$
-mY^2-\beta Y+\alpha X-mXY=\frac{2m(m+1)}{(m-1)^2}X[X(P_2)-X],
$$
which is positive in the region $\{0\leq X<X(P_2)\}$. The second curve we consider is the curve where $dY/dX=0$, that is, of equation
$$
-mY^2-\beta Y+\alpha X-mXY=0,
$$
and the flow of the system \eqref{PPsystZ0} across this curve is given by the sign of the expression
$$
mX(\alpha-mY)[(m-1)Y-2X],
$$
which is negative in the region $Y<Y(P_2)$ since $\alpha>1/(m+1)=mY(P_2)$. Since both curves above connect $P_0$ to $P_2$, they bound an interior region in which an orbit of the phase plane associated to the system \eqref{PPsystZ0} may enter from outside, but never go out of it. Since the point $P_2$ is an attractor for the phase plane associated to \eqref{PPsystZ0}, it is an easy verification that the orbits going out of $P_0$ inside the region enclosed by the two curves will enter $P_2$. The same proof for the case $p=1$ is given in great detail as Lemma 5.4 in \cite{IS1}.
\end{proof}
Let us next introduce:
$$
Y_0:=\frac{(m-1)(\sigma+2)}{2m[\sigma(m-1)+2(p-1)]}=\frac{(m-1)(\sigma+2)}{2(m-p)}\frac{\beta}{m}, \quad p\in(1,m), \ \sigma>0.
$$
Notice first that $Y_0>\beta/m$ for any $p>1$, $\sigma>0$. We then have
\begin{lemma}\label{lem.flow}
If an orbit in the phase space associated to the system \eqref{PSsyst2} crosses the plane $\{Y=-Y_0\}$ for $p\in(1,m)$, $\sigma>0$ and at the crossing point the coordinate $X<X(P_2):=(m-1)/2m(m+1)$, then it cannot reenter the half-space $\{Y>-Y_0\}$ later on.
\end{lemma}
\begin{proof}
The flow of the system \eqref{PSsyst2} on the plane $\{Y=-Y_0\}$ is given by the sign of $P(-Y_0)$, where
$$
P(Y):=-mY^2-(\beta+mX)Y+\alpha X-XZ,
$$
whose smallest real root (if any, otherwise $P(y)<0$ for any $y\in\real$ and the conclusion becomes obvious) is
$$
\overline{y}:=-\frac{(\beta+mX)+\sqrt{(\beta+mX)^2+4m(\alpha X-Z)}}{2m}.
$$
We readily get that
$$
\overline{y}>-\frac{(\beta+mX(P_2))+\sqrt{(\beta+mX(P_2))^2+4m\alpha X(P_2)}}{2m}=-Y_0,
$$
whence $P(-Y_0)<0$. Since in the half-space $\{Y<0\}$ we have $\dot{X}<0$, the inequality $X<X(P_2)$ is also preserved along the orbit after the crossing point, thus the orbit cannot come back and cross again the plane $\{Y=-Y_0\}$.
\end{proof}
\noindent \textbf{Remark.} It readily follows from Lemma \ref{lem.flow} and the inequality \eqref{interm28} that if $f$ is a profile with interface, then its orbit in the phase space does not cross the plane $\{Y=-Y_0\}$ and there exists a constant $C>0$ depending only on $m$ and $p$, but \emph{independent of $\sigma$ and $f$}, such that
\begin{equation}\label{unif.deriv}
|(f^{m-1})'(\xi)|\leq C\xi, \qquad {\rm for \ any} \ \xi>0.
\end{equation}
Moreover,
\begin{equation}\label{unif.deriv2}
(f^m)'(\xi)=\frac{m}{m-1}(f^{m-1})'(\xi)f(\xi)\leq K(m,p)\xi^{(m+1)/(m-1)},
\end{equation}
which is an immediate consequence of \eqref{upper} and \eqref{unif.deriv}, the constant $K(m,p)$ depending only on $m$ and $p$.

We are now in a position to state the main result concerning the good profiles contained in the orbit starting from the critical points $P_2$ and $P_0$ for sufficiently large $\sigma>0$.
\begin{proposition}\label{prop.P2large}
Given $p\in(1,m)$ fixed, there is $\sigma_1>0$ (depending on $p$) such that, for any $\sigma\in(\sigma_1,\infty)$, the orbit going out of $P_2$ in the phase space associated to the system \eqref{PSsyst2} enters the critical point at infinity denoted by $Q_3$ on the Poincar\'e hypersphere. Moreover, for any $\sigma\in(\sigma_1,\infty)$ there are orbits connecting $P_0$ and $Q_3$ in the same phase space.
\end{proposition}
\begin{proof}
The proof is technically involved and consists in constructing suitable geometric barriers in form of planes in the phase space associated to the system \eqref{PSsyst2}, limiting the way for the orbit coming out of $P_2$ (at least for suitably large $\sigma$). Some of the calculations required in the proof are very technical and were performed with the help of a computer program. We divide the proof into several steps.

\medskip

\noindent \textbf{Step 1.} Consider the following two planes of equations
\begin{equation}\label{plane1}
Z=E-DY, \quad D=\frac{2m(m+1)^2}{m-1}, \ E=\frac{2(m+1)}{m-1},
\end{equation}
and respectively
\begin{equation}\label{plane2}
X=BY+C, \quad B=\frac{m(m-1)}{2m^2+5m+1}, \ C=\frac{(2m+1)(m-1)}{2m(2m^2+5m+1)}
\end{equation}
Notice that both planes contain the point $P_2$. After some calculations, the flow of the system \eqref{PSsyst2} on the plane given by \eqref{plane1} is given by the sign of the following complicated expression:
\begin{equation*}
\begin{split}
F(X,Y&;m,p,\sigma):=-mpY^2-\frac{m^2+m\sigma-mp-mp\sigma+m+p\sigma-2p^2+3p-\sigma-2}{(m+1)(\sigma(m-1)+2(p-1))}Y\\
&+\frac{m(2m^2-m\sigma+3m+\sigma+3)}{m-1}XY-\frac{L(m,p,\sigma)}{(m+1)(m-1)(\sigma(m-1)+2(p-1))}X,
\end{split}
\end{equation*}
where
$$
L(m,p,\sigma):=-(m-1)^2\sigma^2+\sigma(m-1)(2m^2+3m+3-2p)+4p(m+1)^2-2(m+1)(3m+1).
$$
We show first that in the region of the phase space where $-Y_0<Y<Y(P_2):=1/m(m+1)$, the term multiplying $X$ in the expression of $F(X,Y;m,p,\sigma)$ is positive, that is
\begin{equation}\label{interm37}
K(Y):=\frac{m(2m^2-m\sigma+3m+\sigma+3)}{m-1}Y-\frac{L(m,p,\sigma)}{(m+1)(m-1)(\sigma(m-1)+2(p-1))}>0,
\end{equation}
for $\sigma>0$ sufficiently large. Indeed, we notice that the coefficient of $Y$ in \eqref{interm37} is negative for $\sigma>(2m^2+3m+3)/(m-1)$ and the same happens for the expression $L(m,p,\sigma)$ for $\sigma$ large enough. Thus it is obvious that $K(Y)>0$ for $Y\leq0$. Moreover, replacing $Y=Y(P_2)=1/m(m+1)$ in \eqref{interm37}, we get
$$
K(Y(P_2))=\frac{2(m-p+\sigma+2)}{(m+1)(\sigma(m-1)+2(p-1))}>0.
$$
Since $K(Y)$ is a linear expression, it follows that $K(Y)>0$ for $Y\in(-Y_0,Y(P_2))$. Let us restrict now to the region of the plane \eqref{plane1} limited by the intersection with the plane \eqref{plane2} in which we have $X>BY+C$, where $B$, $C$ are given in \eqref{plane2}. According to the positivity of $K(Y)$, we infer that for $-Y_0<Y<Y(P_2)$ and $\sigma>0$ sufficiently large holds
\begin{equation*}
\begin{split}
F(X,Y;m,p,\sigma)&>F(BY+C,Y;m,p,\sigma)\\&=\frac{(m(m+1)Y-1)(K_1Y+K_2)}{2m(m+1)(2m^2+5m+1)(\sigma(m-1)+2(p-1))},
\end{split}
\end{equation*}
where
\begin{equation*}
\begin{split}
K_1&=2m(\sigma(m-1)+2(p-1))\left[-m(m-1)\sigma+2m^3+3m^2+3m-2m^2p-5mp-p\right],\\
K_2&=(2m+1)\left[-(m-1)^2\sigma+(m-1)(2m^2+3m-2p+3)\sigma\right.\\&\left.+2(m+1)(2mp-3m+2p-1)\right]
\end{split}
\end{equation*}
One can next prove (we omit here the rather tedious but straightforward calculations, for example by estimating at $Y=-Y_0$, $Y=0$ and noticing that $K_1<0$ for $\sigma>0$ large) that for $\sigma>0$ large enough, the quantity $K_1Y+K_2$ is negative for $-Y_0<Y<Y(P_2)$. Moreover, since $Y<Y(P_2)$, it follows that $m(m+1)Y<1$, hence $F(BY+C,Y;m,p,\sigma)>0$ in the region we are concerned with, that is $-Y_0<Y<Y(P_2)$. Thus, the flow in this region of the plane \eqref{plane1} is in the positive direction of the normal to the plane \eqref{plane1}.

\medskip

\noindent \textbf{Step 2.} We consider now the plane given by the equation \eqref{plane2} and we restrict ourselves to the region where $Z>E-DY$ with $D$, $E$ as in \eqref{plane1}. The flow of the system \eqref{PSsyst2} on the plane \eqref{plane2} is given by the sign of the following expression:
\begin{equation*}
\begin{split}
H(Y,Z;&m,p,\sigma)=\frac{2m^3(m-1)(m+1)^2}{(2m^2+5m+1)^2}Y^2+\frac{(m-1)(m+1)J(m,p,\sigma)}{2(\sigma(m-1)+2(p-1))(2m^2+5m+1)^2}Y\\
&+\left[\frac{m^2(m-1)^2}{(2m^2+5m+1)^2}Y+\frac{(2m+1)(m-1)^2}{2(2m^2+5m+1)^2}\right]Z\\
&-\frac{(2m^2\sigma+4mp-2m+2p-\sigma-2)(2m+1)(m-1)^2}{2m(\sigma(m-1)+2(p-1))(2m^2+5m+1)^2},
\end{split}
\end{equation*}
where
$$
J(m,p,\sigma):=\sigma(4m^4-6m^3+m+1)+(8m^3-8m^2-6m-2)p-4m^3+6m^2+4m+2.
$$
Since the coefficient of $Z$ in the expression of $H(Y,Z;m,p,\sigma)$ is $BY+C=X>0$, we can replace $Z$ in the above mentioned region by $E-DY$ and thus obtain
\begin{equation*}
\begin{split}
H(Y,Z;m,p,\sigma)&>H(Y,E-DY;m,p,\sigma)\\&=-\frac{(m-1)(2m+1)S(m,p,\sigma)}{2m(\sigma(m-1)+2(p-1))(2m^2+5m+1)^2}(m^2Y+mY-1),
\end{split}
\end{equation*}
with
$$
S(m,p,\sigma):=(2m+1)(m-1)\sigma-2m^2+6mp-4m+2p-2>0,
$$
provided $\sigma>0$ is large enough. Noticing that $m^2Y+mY-1=m(m+1)(Y-Y(P_2))$ which is negative on any connection coming out of $P_2$ according to Lemma \ref{lem.flow2}, we readily infer that the flow on the considered region of the plane \eqref{plane2} has positive sign for $Y\in(-Y_0,Y(P_2))$.

\medskip

\noindent \textbf{Step 3. Connections from $P_2$ to $Q_3$.} Let us now take $\sigma>0$ sufficiently large in order to fulfill all the conditions for the positivity of the flows in the previous Steps 1 and 2. Recall that a connection coming out of the critical point $P_2$ starts tangent to the eigenvector $e_3=(x(\sigma),-1,z(\sigma))$, where $x(\sigma)$, $z(\sigma)$ are defined in the proof of Lemma \ref{lem.3}. Noticing that the scalar product between the normal to the plane given by \eqref{plane1} and $e_3$ is
$$
(0,D,1)\cdot e_3=z(\sigma)-D>0,
$$
for $\sigma>0$ large enough, we infer that the orbit coming out of $P_2$ starts in the region where $Z>E-DY$ (and $Y<Y(P_2)$ by Lemma \ref{lem.flow2}). Moreover, the same can be said about the plane given by \eqref{plane2}, since
$$
(1,-B,0)\cdot e_3=x(\sigma)+B=\frac{m(m-1)}{2m^2+5m+1}-\frac{(m-1)^2}{2(m+p-2)+\sigma(m-1)}>0,
$$
provided $\sigma>0$ is very large. Thus, the orbit coming out of $P_2$ also begins in the region where $X>BY+C$ for $\sigma>0$ large enough. But then, the outcome of Steps 1 and 2 above proves that the orbit coming out of $P_2$ has to remain in the region where simultaneously
$$
X>BY+C, \ Z>E-DY, \ {\rm for} \ -Y_0<Y<Y(P_2).
$$
Indeed, due to the positive signs of the flows given by $F(X,Y;m,p,\sigma)$ and $H(Y,Z;m,p,\sigma)$ in the corresponding region, the orbit cannot intersect any of the two planes given by the equations \eqref{plane1} and \eqref{plane2}. This means that, in particular, taking $\sigma>\sigma_1>0$ with $\sigma_1$ large enough so that all the positivity conditions in Steps 1, 2 and 3 are satisfied, the orbit coming out of $P_2$ will remain in the geometric region $\{X>BY+C\}$ while $Y>-Y_0$. It is then obvious that the critical point $P_{\gamma_0}$ does not lie in the region $\{X>BY+C\}$. Moreover, the same is true for the interface critical point $P_1=(0,-\beta/m,0)$ given $\sigma$ large enough, since
$$
-\frac{\beta}{m}B+C=\frac{m-1}{2m^2+5m+1}\left[\frac{2m+1}{2m}-\frac{m-p}{\sigma(m-1)+2(p-1)}\right]>0.
$$
It thus follows that, provided $\sigma>0$ sufficiently large, the orbit coming out of $P_2$ cannot enter any of the critical points $P_{\gamma_0}$ and $P_1$ before, in a first step, crossing the plane $\{Y=-Y_0\}$. But then Lemma \ref{lem.flow} implies that the orbit cannot go back to enter again the half-space $\{Y>-Y_0\}$, in which both critical points $P_{\gamma_0}$ and $P_1$ lie. This argument discards that, for $\sigma>0$ large enough so that all the previous conditions of positivity hold true, the orbit coming out of $P_2$ may enter the points $P_{\gamma_0}$ and $P_1$. From the monotonicity of the components $X$ (see Lemma \ref{lem.flow2}) and $Y$ in the region $\{Y<-Y_0\}$ (see the proof of Lemma \ref{lem.flow}), and from the invariance of the $\omega$-limit set \cite[Theorem 2, Section 3.2]{Pe}, we infer that the orbit coming out of $P_2$ for $\sigma>0$ sufficiently large must enter a critical point in the closure of the region $\{Y<-Y_0\}$, and the only such critical point is $Q_3$.

\medskip

\noindent \textbf{Step 4. Connections from $P_0$ to $Q_3$.} Let $p\in(1,m)$ and $\sigma>\sigma_1$ be fixed, where $\sigma_1>0$ is large enough such that all the positivity conditions in Steps 1, 2, 3 are simultaneously satisfied and thus the orbit coming out of $P_2$ in the phase space connects to the stable node $Q_3$ at infinity. Since $Q_3$ is a stable node and $P_2$ a saddle point, there exists $\delta>0$ sufficiently small such that for any (non-critical) point in a small half-ball near $P_2$, namely $(X,Y,Z)\in B(P_2,\delta)\cap\{Z>0\}$, the unique orbit passing through this point in the phase-space enters $Q_3$. It also follows from Lemma \ref{lem.Z0} that there is an orbit connecting $P_0$ to $P_2$ inside the plane $\{Z=0\}$. Again by continuity, there exists an orbit going out of $P_0$ and entering the ball $B(P_2,\delta)$ (without connecting to $P_2$). We thus infer that this orbit in the phase space coming out of $P_0$ enters $Q_3$, as desired.
\end{proof}

\subsection{Proof of Theorems \ref{th.large}\,(b) and \ref{th.decay}}\label{subsec.55}

We begin with the following preparatory result:
\begin{lemma}\label{lem.P0P1}
(a) Let $m>1$, $p\in(1,m)$ be such that there exist $\sigma_1$, $\sigma_2>0$ such that for $\sigma=\sigma_1$, the orbit coming out of $P_2$ in the phase space associated to the system \eqref{PSsyst2} enters $P_{\gamma_0}$ and for $\sigma=\sigma_2$ the orbit coming out of $P_2$ enters $Q_3$. Then there exists $\sigma_0>0$ such that for $\sigma=\sigma_0$ there is a good profile with interface with behavior \eqref{case1} near $\xi=0$.

(b) Let $m>1$, $p\in(1,m)$ and $\sigma>0$ be such that in the phase space associated to the system \eqref{PSsyst2} there are orbits connecting $P_0$ to $P_{\gamma_0}$ and orbits connecting $P_0$ to $Q_3$. Then there exists at least one good profile with interface with behavior \eqref{case2} near $\xi=0$.
\end{lemma}
\begin{proof}
(a) We define the following three sets:
\begin{equation*}
\begin{split}
&A:=\{\sigma>0: {\rm the \ orbit \ from \ }P_2 \ {\rm enters} \ P_{\gamma_0}\},\\
&B:=\{\sigma>0: {\rm the \ orbit \ from \ }P_2 \ {\rm enters} \ P_1\},\\
&C:=\{\sigma>0: {\rm the \ orbit \ from \ }P_2 \ {\rm enters} \ Q_3\}.
\end{split}
\end{equation*}
Since $P_{\gamma_0}$ and $Q_3$ are both attractors, it is easy to see that the sets $A$ and $C$ are open. Moreover, both sets are non-empty, since $\sigma_1\in A$ and $\sigma_2\in C$. It follows from the analysis in Section \ref{sec.local} that $A\cup B\cup C=(0,\infty)$ and the three sets are obviously disjoint. It then follows that $B$ is nonempty and closed, hence any element $\sigma_0\in B$ gives the conclusion.
(b) Inspecting the proof of Lemma \ref{lem.1}, we readily observe that the orbits going out of $P_0$ tangent to the center manifold form a one-parameter family depeding on the constant $k>0$ such that $Z\sim kX$ near $\xi=0$. We also remark from the local analysis done in Section \ref{sec.local} that the orbits coming out of $P_0$ can only enter the attractors $P_{\gamma_0}$ and $Q_3$ and the interface critical point $P_1$. We can thus define the sets:
\begin{equation*}
\begin{split}
&A_0:=\{k>0: {\rm the \ orbit \ from \ }P_0 \ {\rm with} \ Z\sim kX \ {\rm enter} \ P_{\gamma_0}\},\\
&B_0:=\{k>0: {\rm the \ orbit \ from \ }P_0 \ {\rm with} \ Z\sim kX \ {\rm enter} \ P_1\},\\
&C_0:=\{k>0: {\rm the \ orbit \ from \ }P_0 \ {\rm with} \ Z\sim kX \ {\rm enter} \ Q_3\}.
\end{split}
\end{equation*}
The hypothesis implies that $A_0$ and $C_0$ are nonempty. Moreover, $A_0\cup B_0\cup C_0=(0,\infty)$ and $A_0$ and $C_0$ are open sets, since both $P_{\gamma_0}$ and $Q_3$ are attractors. Thus the complement $B_0$ is a closed and nonempty set, hence there exists at least a connection in the phase space from $P_0$ to $P_1$ as claimed.
\end{proof}
We are now in a position to prove part (b) of Theorem \ref{th.large}.
\begin{proof}[Proof of Theorem \ref{th.large},\,(b)]
Let $p\in(1,m)$ and $\sigma\in(\sigma_1,\infty)$, with $\sigma_1>0$ defined in Proposition \ref{prop.P2large}. Then Proposition \ref{prop.P2large} implies that the orbit starting from $P_2$ enters $Q_3$ and there are orbits going out of $P_0$ and entering $Q_3$. We also infer from Lemma \ref{lem.tail} that there are orbits going out of $P_0$ and entering $P_{\gamma_0}$. Lemma \ref{lem.P0P1} then proves that there are orbits coming out of $P_0$ and entering $P_1$, for any $p\in(1,m)$ and $\sigma\in(\sigma_1,\infty)$. Any profile contained in such an orbit is a good profile with interface with behavior near $\xi=0$ given by \eqref{case2}.
\end{proof}
Since we are also interested in profiles with decay \eqref{tail} as $\xi\to\infty$, we prove first the following result for $\sigma=0$.
\begin{lemma}\label{lem.homogeneous}
In the homogeneous case $\sigma=0$, all the good profiles with $f(0)=0$ and $(f^m)'(0)=0$, go to a constant as $\xi\to\infty$, that is
$$
\lim\limits_{\xi\to\infty}f(\xi)=\left(\frac{1}{p-1}\right)^{1/(p-1)}.
$$
\end{lemma}
\begin{proof}
We adapt here an argument from \cite[Lemma 1, pp. 183-184]{S4} (in
the cited book, it is used for profiles with $f(0)=a>0$ and
$f'(0)=0$). We multiply Eq. \eqref{SSODE} (with $\sigma=0$) by
$f^{m-1}(\xi)f'(\xi)$ to get
\begin{equation}\label{interm24}
\begin{split}
(m-1)(f^{m-1}f')(\xi)(f^{m-1}f')'(\xi)&-\alpha(f^{m}f')(\xi)+\beta\xi
f^{m-1}(\xi)(f')^2(\xi)\\&+(f^{m+p-1}f')(\xi)=0.
\end{split}
\end{equation}
Integrating \eqref{interm24} on a generic interval $(0,\xi)$ and
taking into account the initial conditions $f(0)=0$, $(f^m)'(0)=0$,
we get
\begin{equation}\label{interm25}
\begin{split}
\frac{1}{m+p}f^{m+p}(\xi)&-\frac{\alpha}{m+1}f^{m+1}(\xi)+\frac{m-1}{2}(f^{m-1}f')^2(\xi)\\&=-\beta\int_{0}^{\xi}sf^{m-1}(s)(f')^2(s)\,ds<0.
\end{split}
\end{equation}
This shows that there is no point $\xi_0>0$ such that $f(\xi_0)=0$,
as the existence of such point would contradict \eqref{interm25}. We
thus infer that $f(\xi)>0$ for any $\xi>0$. But coming back to the
list of behaviors in Section \ref{sec.local} (which at the level of
the phase space also works for $\sigma=0$), the only possibility is
that the profiles $f$ belong to orbits entering the attractor
$P_{\gamma_0}$, that is, go to the special constant
$(1/(p-1))^{1/(p-1)}$ as $\xi\to\infty$, as stated.
\end{proof}
This helps us to complete the proof of our Theorem \ref{th.decay}.
\begin{proof}[Proof of Theorem \ref{th.decay}] As the point $P_{\gamma_0}$ is an attractor for any $\sigma\geq0$ and as for $\sigma=0$, the orbit going out of $P_2$ and all orbits going out of $P_0$ enter $P_{\gamma_0}$ according to Lemma \ref{lem.homogeneous}, standard continuity arguments show that this fact stays true in a right neighborhood $\sigma\in(0,\sigma_0)$ of $\sigma=0$. We stress here that, although at the level of the phase space is the same behavior, at the level of profiles a big difference occur when passing from $\sigma=0$ to $\sigma>0$: in the former, these profiles were converging to the constant $(1/(p-1))^{1/(p-1)}$ as $\xi\to\infty$, while in the latter they decay to zero with the decay rate given in \eqref{queue}. Finally, Lemma \ref{lem.tail} shows that for any $\sigma>0$ there exists at least an orbit connecting $P_0$ and $P_{\gamma_0}$ in the phase space, concluding the proof.
\end{proof}

\subsection{Critical connections. Proof of Theorem \ref{th.large}\,(a)}

With the preparations done in the previous subsections, we are now in a position to proceed to the proof of our remaining result, that is the existence of \emph{good profiles with interface} with behavior as in \eqref{case1} starting with $f(0)=0$, $(f^m)'(0)=0$, for any $p\in(1,m)$ and some $\sigma^*=\sigma^*(m,p)>0$.
\begin{proof}[Proof of Theorem \ref{th.large},\,(a)]
Fix $m>1$. On the one hand, it follows in particular from the above proof of Theorem \ref{th.decay} that for any $p\in(1,m)$, there exists $\sigma_0>0$ such that the orbit starting from $P_2$ for $\sigma\in(0,\sigma_0)$ enters $P_{\gamma_0}$. On the other hand, we infer from Proposition \ref{prop.P2large} that for any $p\in(1,m)$ and $\sigma>\sigma_1>0$ sufficiently large, the orbit coming out of $P_2$ enters the critical point $Q_3$. Thus, by Lemma \ref{lem.P0P1}, part (a), we deduce that for any $p\in(1,m)$ there exists at least a connection from $P_2$ to $P_1$ for some $\sigma=\sigma^{*}(p)$, hence a \emph{good profile with interface} $f$ starting with $f(0)=0$, $(f^m)'(0)=0$, as claimed.
\end{proof}

For the reader's convenience, we gather in Figure \ref{fig10} below a visual representation of how the connections starting from the points $P_2$ and $P_0$ in the phase space associated to the system \eqref{PSsyst2} vary with $\sigma>0$. We notice how the good profiles with interface arrive from these two points according to $\sigma$, as proved above. The numerical experiments were realized with $m=3$, $p=1.5$ and for the three cases $\sigma=1.5$ (small), $\sigma=2.3218$ (critical) and $\sigma=3$ (large).

\begin{figure}[ht!]
  \begin{center}
  \subfigure[$\sigma$ small]{\includegraphics[width=7cm,height=5cm]{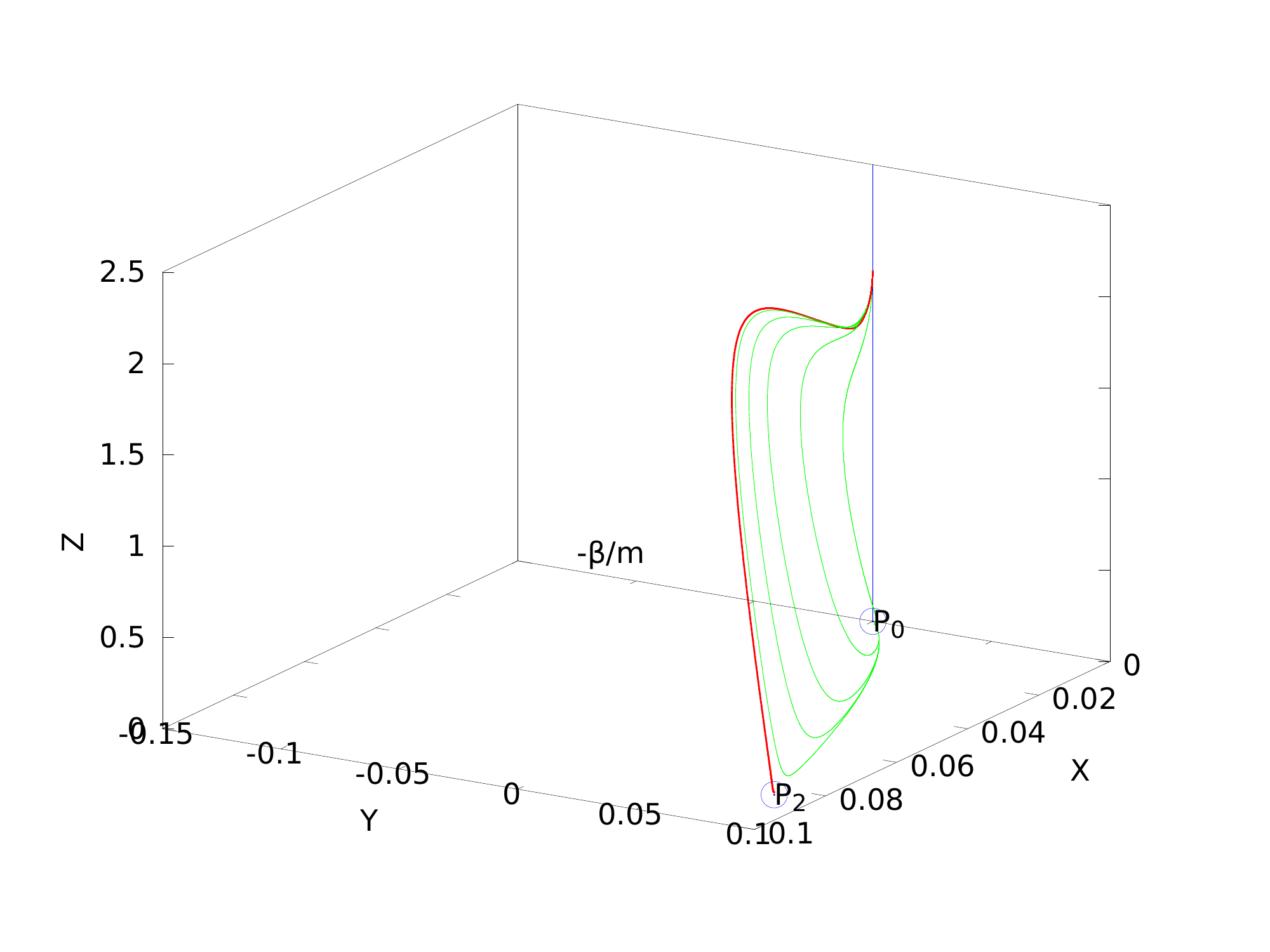}}
  \subfigure[Critical $\sigma^*$]{\includegraphics[width=7cm,height=5cm]{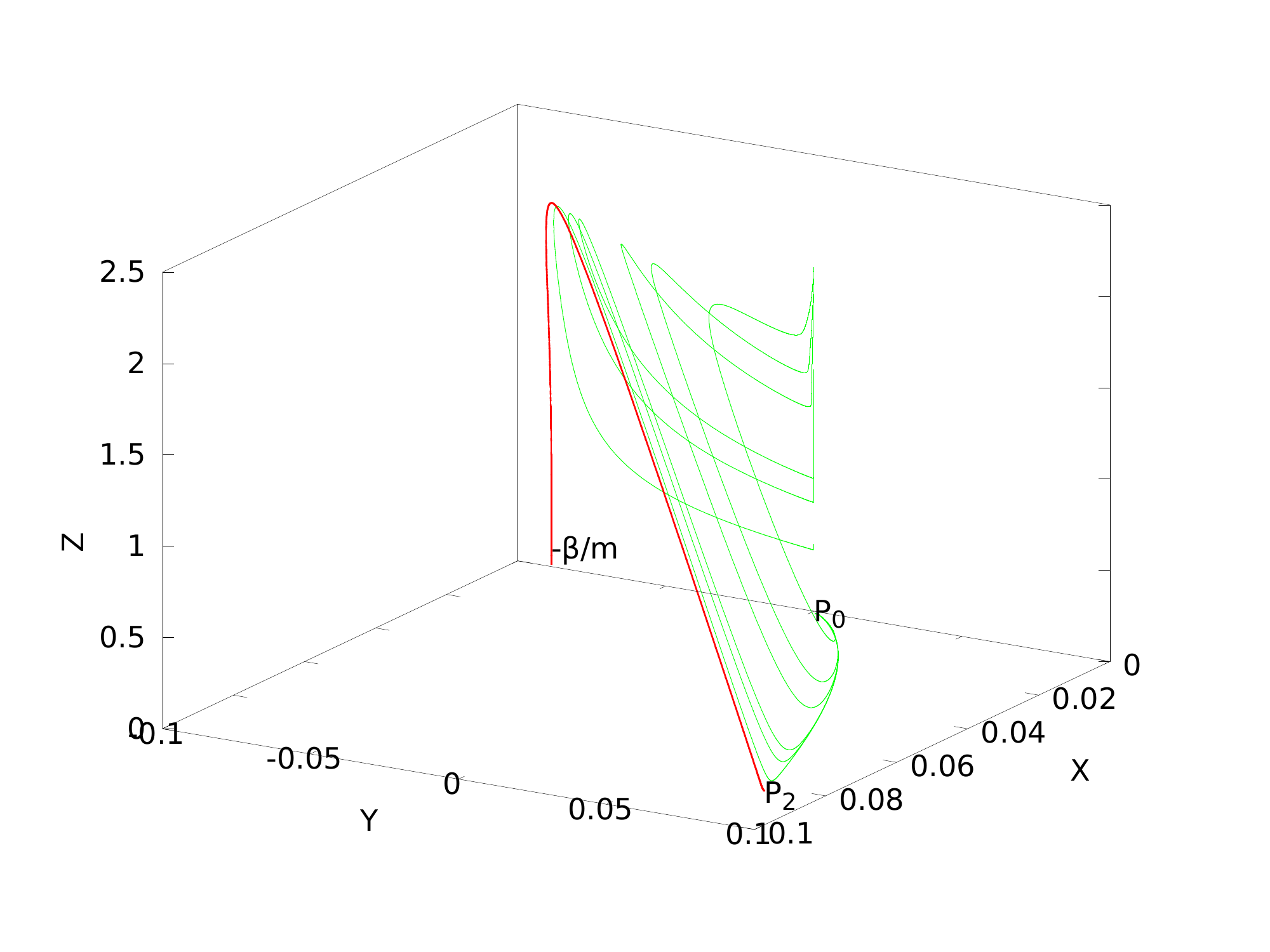}}
  \subfigure[$\sigma$ large]{\includegraphics[width=7cm,height=5cm]{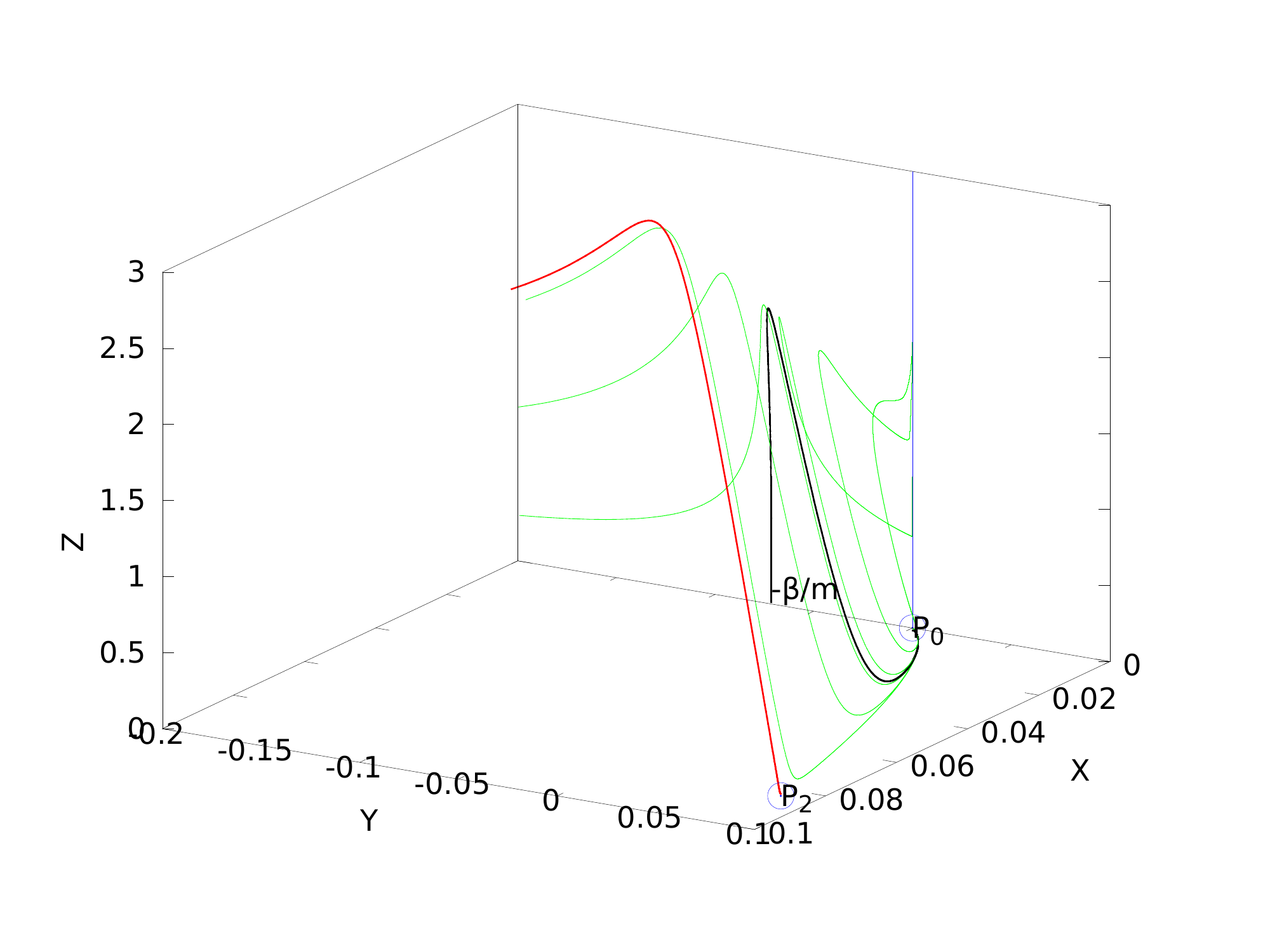}}
  \end{center}
  \caption{Orbits from $P_2$ and $P_0$ for different values of $\sigma$}\label{fig10}
\end{figure}

\section*{Acknowledgements} R. I. is supported by the ERC Starting Grant GEOFLUIDS 633152. A. S. is partially supported by the Spanish project MTM2017-87596-P.

\bibliographystyle{plain}

\begin{thebibliography}{1}

\bibitem{AT05}
D. Andreucci, and A. F. Tedeev, \emph{Universal bounds at the
blow-up time for nonlinear parabolic equations}, Adv. Differential
Equations, \textbf{10} (2005), no. 1, 89-120.

%\bibitem{AV95}
%D. G. Aronson, and J. L. V\'azquez, \emph{Anomalous exponents in nonlinear diffusion}, J. Nonlinear Sci., \textbf{5} (1995), no. 1, 29-56.

\bibitem{BZZ11}
X. Bai, S. Zhou, and S. Zheng, \emph{Cauchy problem for fast
diffusion equation with localized reaction}, Nonlinear Anal.,
\textbf{74} (2011), no. 7, 2508-2514.

\bibitem{BL89}
C. Bandle, and H. Levine, \emph{On the existence and nonexistence of
global solutions of reaction-diffusion equations in sectorial
domains}, Trans. Amer. Math. Soc., \textbf{316} (1989), 595-622.

\bibitem{BK87}
P. Baras, and R. Kersner, \emph{Local and global solvability of a
class of semilinear parabolic equations}, J. Differential Equations,
\textbf{68} (1987), 238-252.

\bibitem{BHK01}
M. Bowen, J. Hulshof, and J. R. King, \emph{Anomalous exponents and dipole solutions for the thin film equation}, SIAM J. Appl. Math., \textbf{62} (2001), no. 1, 149-179.

\bibitem{CQW03}
X. Chen, Y. Qi, and M. Wang, \emph{Self-similar singular solutions
of a $p$-Laplacian evolution equation with absorption}, J.
Differential Equations, \textbf{190} (2003), 1-15.

\bibitem{FdPV06}
R. Ferreira, A. de Pablo, and J. L. V\'azquez, \emph{Classification
of blow-up with nonlinear diffusion and localized reaction}, J.
Differential Equations, \textbf{231} (2006), no. 1, 195-211.

\bibitem{FV}
R. Ferreira, and J. L. V\'azquez, \emph{Extinction behaviour for
fast diffusion equations with absorption}, Nonlinear Anal.,
\textbf{43} (2001), no. 8, 943-985.

\bibitem{GV}
V. A. Galaktionov, and J. L. V\'azquez, \emph{Continuation of blowup
solutions of nonlinear heat equations in several space dimensions},
Comm. Pure Appl. Math, \textbf{50} (1997), no. 1, 1-67.

\bibitem{GU05}
Y. Giga, and N. Umeda, \emph{Blow-up directions at space infinity for solutions of semilinear heat equations}, Bol. Soc. Paran. Mat., \textbf{23} (2005), 9-28.

\bibitem{GU06}
Y. Giga, and N. Umeda, \emph{On blow-up at space infinity for semilinear heat equations}, J. Math. Anal. Appl., \textbf{316} (2006), 538-555.

\bibitem{GP76}
B. H. Gilding, and L. A. Peletier, \emph{On a class of similarity solutions of the porous media equation}, J. Math. Anal. Appl., \textbf{55} (1976), 351-364.

\bibitem{GLS}
J.-S. Guo, C.-S. Lin, and M. Shimojo, \emph{Blow-up behavior for a
parabolic equation with spatially dependent coefficient}, Dynam.
Systems Appl., \textbf{19} (2010), no. 3-4, 415-433.

\bibitem{GS11}
J.-S. Guo, and M. Shimojo, \emph{Blowing up at zero points of
potential for an initial boundary value problem}, Commun. Pure Appl.
Anal., \textbf{10} (2011), no. 1, 161-177.

\bibitem{GLS13}
J.-S. Guo, C.-S. Lin, and M. Shimojo, \emph{Blow-up for a
reaction-diffusion equation with variable coefficient}, Appl. Math.
Lett., \textbf{26} (2013), no. 1, 150-153.

\bibitem{HS}
M. W. Hirsch, and S. Smale, \emph{Differential equations, dynamical
systems, and linear algebra}, Pure and Applied Mathematics, vol. 60,
Academic Press, New York-London, 1974.

\bibitem{IL13a}
R. G. Iagar, and Ph.~Lauren\ced{c}ot, \emph{Existence and uniqueness
of very singular solutions for a fast diffusion equation with
gradient absorption},  J. London Math. Soc., \textbf{87} (2013),
509-529.

\bibitem{IL13b}
R. G. Iagar, and Ph. Lauren\ced{c}ot, \emph{Eternal solutions to a
singular diffusion equation with critical gradient absorption},
Nonlinearity, \textbf{26} (2013), no. 12, 3169-3195.

\bibitem{IS1}
R. Iagar, and A. S\'anchez, \emph{Blow up profiles for a quasilinear reaction-diffusion equation with weighted diffusion and linear growth}, J. Dynamics Differential Equations, to appear (accepted December 2018).

\bibitem{ISV}
R. G. Iagar, A. S\'anchez, and J. L. V\'azquez, \emph{Radial
equivalence for the two basic nonlinear degenerate diffusion
equations}, J. Math. Pures Appl., \textbf{89} (2008), no. 1, 1-24.

\bibitem{IU08}
T. Igarashi, and N. Umeda, \emph{Existence and nonexistence of
global solutions in time for a reaction-diffusion system with
inhomogeneous terms}, Funkcial. Ekvac., \textbf{51} (2008), no. 1,
17-37.

\bibitem{KWZ11}
X. Kang, W. Wang, and X. Zhou, \emph{Classification of solutions of
porous medium equation with localized reaction in higher space
dimensions}, Differential Integral Equations, \textbf{24} (2011),
no. 9-10, 909-922.

\bibitem{La84}
A. A. Lacey, \emph{The form of blow-up for nonlinear parabolic equations}, Proc. Royal Society Edinburgh Sect. A, \textbf{98} (1984), no. 1-2, 183-202.

\bibitem{Liang12}
Z. Liang, \emph{On the critical exponents for porous medium equation
with a localized reaction in high dimensions}, Commun. Pure Appl.
Anal., \textbf{11} (2012), no. 2, 649-658.

\bibitem{Ly51}
L. S. Lyagina, \emph{The integral curves of the equation $y'=\frac{ax^2+bxy+cy^2}{dx^2+exy+fy^2}$}, Uspekhi Mat. Nauk, \textbf{6} (1951), no. 2, 171-183 (Russian).

%\bibitem{dPS98}
%A. de Pablo, and A. S\'anchez, \emph{Travelling wave behavior for a Porous-Fisher equation}, European J. Appl. Math., \textbf{9} (1998), no. 3, 285-304.

\bibitem{dPS00}
A. de Pablo, and A. S\'anchez, \emph{Global travelling waves in reaction-convection-diffusion equations}, J. Differential Equations, \textbf{165} (2000), no. 2, 377-413.

\bibitem{Pe}
L. Perko, \emph{Differential equations and dynamical systems. Third
edition}, Texts in Applied Mathematics, \textbf{7}, Springer Verlag,
New York, 2001.

\bibitem{Pi97}
R. G. Pinsky, \emph{Existence and nonexistence of global solutions
for $u_t=\Delta u+a(x)u^p$ in $\real^d$}, J. Differential Equations,
\textbf{133} (1997), no. 1, 152-177.

\bibitem{Pi98}
R. G. Pinsky, \emph{The behavior of the life span for solutions to
$u_t=\Delta u+a(x)u^p$ in $\real^d$}, J. Differential Equations,
\textbf{147} (1998), no. 1, 30-57.

\bibitem{QS}
P. Quittner, and Ph. Souplet, \emph{Superlinear parabolic problems.
Blow-up, global existence and steady states}, Birkhauser Advanced
Texts, Birkhauser Verlag, Basel, 2007.

\bibitem{S4}
A. A. Samarskii, V. A. Galaktionov, S. P. Kurdyumov, and A. P.
Mikhailov, \emph{Blow-up in quasilinear parabolic problems}, de
Gruyter Expositions in Mathematics, \textbf{19}, W. de Gruyter,
Berlin, 1995.

\bibitem{Su02}
R. Suzuki, \emph{Existence and nonexistence of global solutions of
quasilinear parabolic equations}, J. Math. Soc. Japan, \textbf{54}
(2002), no. 4, 747-792.

\end{thebibliography}

\end{document}